\documentclass[10pt]{amsart}

\usepackage{amssymb} 
\usepackage{amsthm}
\usepackage{amsmath}
\usepackage{amscd}
\usepackage{mathrsfs}
\usepackage{hyperref}
\usepackage{tikz-cd}
\usetikzlibrary{arrows}
\usepackage[english]{babel}
\usepackage{stackrel}
\usepackage[new]{old-arrows}
\usepackage{stmaryrd}
\usepackage{xcolor}
\usepackage[symbols,nogroupskip,sort=none]{glossaries-extra}

\newtheorem{theorem}{Theorem}[section]
\newtheorem*{theorem*}{Theorem}
\newtheorem{lemma}[theorem]{Lemma}
\newtheorem*{lemma*}{Lemma}
\newtheorem{corollary}[theorem]{Corollary}
\newtheorem{proposition}[theorem]{Proposition}
\newtheorem*{proposition*}{Proposition}
\theoremstyle{definition}
\newtheorem{definition}{Definition}[subsection]

\newtheorem{example}[theorem]{Example}
\newtheorem{remark}[theorem]{Remark}

\DeclareMathOperator{\Hom}{Hom} 
\DeclareMathOperator{\Ext}{Ext} 

\DeclareMathOperator{\Der}{Der} 
\DeclareMathOperator{\IDer}{IDer} 
\DeclareMathOperator{\HH}{HH}
\DeclareMathOperator{\im}{im}

\DeclareMathOperator{\coker}{coker}
\DeclareMathOperator{\id}{id}
\DeclareMathOperator{\ev}{ev}
\DeclareMathOperator{\Fix}{Fix}

\newcommand{\uHH}{\underline{\HH}}

\newcommand{\ud}{\underline{d}}
\newcommand{\upartial}{\underline{\partial}}
\newcommand{\uExt}{\underline{\Ext}}

\newcommand{\epi}{\twoheadrightarrow}
\newcommand{\mono}{\rightarrowtail}

\newcommand{\xmono}[1]{\overset{#1}\rightarrowtail}
\newcommand{\ot}[1]{^{\otimes#1}}
\newcommand{\E}{\mathscr{E}}
\newcommand{\C}{\mathscr{C}}
\newcommand{\M}{\text{-Mod}}
\newcommand{\Ab}{\text{Ab}}
\newcommand{\handbook}{\cite{handbook}\text{ Proposition }6.10.9}
\newcommand{\handbooktwo}{\cite{handbook}\text{ Proposition }6.10.2}

\newcounter{Bx}
\newcommand{\itemB}{%
    \addtocounter{Bx}{1}
    \item[B\theBx.]}

\title{Hochschild cohomology in toposes}
\author{Cameron Michie}
\address{Cameron Michie\\
         School of Mathematical Sciences\\
  	Queen Mary University of London\\
         London, E1 4NS\\
        United Kingdom}
\email{c.michie@qmul.ac.uk}

\author{Ivan Toma{\v s}i{\'c}} 
\address{Ivan Toma{\v s}i{\'c}\\
         School of Mathematical Sciences\\
  	Queen Mary University of London\\
         London, E1 4NS\\
        United Kingdom}
\email{i.tomasic@qmul.ac.uk}
\thanks{Supported by EPSRC Standard grant EP/V028812/1}

\date{\today}

\begin{document}

\begin{abstract}
    We develop a theory of \emph{internal Hochschild cohomology} in a ringed topos. We construct it via the internal Hochschild cochain complex, as well as through derived functor/topos cohomology theory, and discuss its relationship to the absolute Hochschild cohomology.
    
By specialising to the topos of difference sets, we obtain a theory of internal \emph{difference Hochschild cohomology}, and compare it to the absolute Hochschild cohomology through the Grothendieck and hypercohomology spectral sequences.

We provide a systematic and detailed treatment of tensor products in suitable toposes in hope to complete the existing literature.  

\end{abstract}

\maketitle





\section{Introduction}

It was shown in \cite[II 6.7]{tome1} that the category
$$
k\M
$$
of unitary left $k$-modules in a ringed Grothendieck topos $(\E,k)$ is an \emph{abelian category} admitting a set of generators and satisfying axioms (AB 5), (AB 3). Hence it has enough injectives and lends itself to the development of derived functors and \emph{topos cohomology}.  

Moreover, \cite[IV 12.2, 12.7]{tome1} consider the internal $k$-hom object $$[N,P]_k\in k\M$$ for $N,P\in k\M$ as a suitable subobject of the topos-theoretic internal hom $[N,P]\in\E$, and then construct the tensor product $$-\otimes_kN$$
as the left adjoint of the functor $[N,-]_k$. 

In the case of a commutative ring $k$, $k\M$ is \emph{monoidal closed} with this structure,  
i.e., we obtain the hom-tensor adjunction
$$
\Hom_k(M\otimes_kN,P)\cong\Hom_k(M,[N,P]_k)
$$ 
for all $k$-modules $M,N$ and $P$ in $\E$.


Our objective is to develop a theory of Hochschild cohomology in a ringed topos $(\E,k)$. Given a $k$-algebra $A$ and an $A$-bimodule $M$, 
we define the \emph{internal Hochschild cohomology} of $A$ with coefficients in $M$ as the cohomology 
$$\underline{\HH}^\bullet(A,M)=H^\bullet C[A;M]$$
of the complex of $k$-modules $C[A;M]$ defined as

    $$0\rightarrow M\xrightarrow{\underline{d}^0}[A,M]_k\xrightarrow{\underline{d}^1}[A^{\otimes2},M]_k\xrightarrow{\underline{d}^2}[A^{\otimes3},M]_k\xrightarrow{\underline{d}^3}...$$

In order to prove that this is a complex and to verify the properties of the cohomology complex, we observed that it is essential to be able to argue by elements in the internal logic of the topos $\E$. Moreover, due to the fact that the objects involved use higher-order constructions, no meta-theorem such as the familiar topos-theoretic principle regarding classical and intuitionistic reasoning in geometric logic (Barr's theorem) can be of use to us. 

On the other hand, Grothendieck's construction of tensor products stems from general category theory principles such as the adjoint functor theorem and predates the development of internal logic, so we decided to provide a detailed treatment of tensor products in toposes in our Appendix, and demonstrate that constructions using universal properties and general categorical arguments agree with the realisations of suitable formulae in the internal language, and that it is justified to argue by elements analogously to the classical proofs.  

We show in \ref{prop-hh-ext} that, for an \emph{enriched projective} $k$-module $A$, Hochschild cohomology can be calculated through derived functors of the internal Hom as 
$$\uHH^n(A,M)\cong F\uExt^n_{A^e}(A,M),$$
where $F$ is the forgetful functor restricting the structure of $A^e$-modules to $k$. 

Motivated by this formula, we define the \emph{absolute}
Hochschild cohomology as
$$
\HH^n(A,M)=\Ext^n_{A^e}(A,M).
$$

In Section~\ref{s-hh-diff}, by specialising the topos to \emph{difference sets}, i.e., by setting
$$
\E=\boldsymbol{\sigma}\textbf{-Set}
$$
we obtain \emph{difference Hochschild cohomology}, calculate its low-degree terms, and give some explicit examples. 

We are able to compute the relationship between internal and absolute difference Hochschild cohomology fairly explicitly as follows. 
The canonical geometric morphism 
$$
\boldsymbol{\sigma}\textbf{-Set}\to \textbf{Set},
$$
has the fixed points functor $\text{Fix}$ as direct image functor, and its derived functors vanish above degree 1. Hence, for an enriched projective $k$-module $A$, we show in \ref{th-groth-ss} that Grothendieck spectral sequence degenerates into short exact sequences
$$
0\to\uHH^{n-1}(A,M)_{\sigma}\to\HH^n_{A^e}(A,M)\to\uHH^n(A,M)^\sigma\to0,
$$
for $n>0$, and the hypercohomology spectral sequences degenerates into 
$$
\cdots\to\HH^n(A,M)\to\uHH^{n-1}_\sigma(A,M)\to\uHH^{n+1,\sigma}(A,M)\to\HH^{n+1}(A,M)\to\cdots,
$$ 
where 
$$\uHH^{n,\sigma}(A,M)$$ is the cohomology of $\Fix$ of the complex $C[A;M]$, and $$\uHH^n_{\sigma}(A,M)$$ is the cohomology of the coinvariants of the complex.

\newpage

\section{Homological algebra internal to a topos}

\subsection{Symmetric monoidal closed structure of \texorpdfstring{$k\M$}{k-Mod}}

Given a topos $\E$, we write $$\text{Ab}(\E)$$ as the category of abelian group objects (and their group homomorphisms) in $\E$, and $$\text{Ring}(\E)$$ as the category of ring objects (and their ring homomorphisms) in $\E$. 

\begin{definition}
    For ringed topos $(\E,A)$, define $$A\M:=\text{Mod}(\E,A)$$ to be the category of left $A$-modules (and their homomorphisms). Similarly, we define $$\text{Mod-}A:=\text{Mod}(\E,A^\text{op})$$ to be the category of right $A$-modules (and their homomorphisms).
\end{definition}

If $(\E,k)$ is a commutative ringed topos, then $k\M\cong\text{Mod-}k$. 

\begin{definition}
    For ringed topos $(\E,A)$, define $$A\text{-Bimod}$$ to be the category of $A$-bimodules.
\end{definition}

\begin{definition}
    For commutative ringed topos $(\E,k)$, define $$k\text{-Alg}$$ to be the category of $k$-algebras, i.e. the category of monoid objects in $k\M$. 
\end{definition}



It is well known that any given topos $\E$ is symmetric monoidal cartesian closed, under the cartesian (categorical) product $\times$ and the internal hom $[-,-]$. By Proposition \ref{ab-functor}, we can define an internal abelian group hom functor  $$[-,-]_\textbf{Z}:\Ab(\E)^\text{op}\times\Ab(\E)\to\Ab(\E).$$ Under this structure, Ab$(\E)$ is a closed category (Proposition \ref{ab-closed}). 

Similarly, let $(\E,A)$ be a ringed topos. By Proposition \ref{mod-functor}, we can define an internal $A$-hom functor  $$[-,-]_A:A\M^\text{op}\times A\M\to\Ab(\E).$$
Given a commutative ringed topos $(\E,k)$, this descends to a functor $$[-,-]_k:k\M^\text{op}\times k\M\to k\M,$$ by Proposition \ref{com-mod-functor}. Under this structure, $k\M$ is a closed category (Proposition \ref{mod-closed}). If $\E$ is such that Ab$(\E)$ has a set of cogenerators (this condition is discussed on page \pageref{rep}), then the functor $$\text{Ab}(\E)\to\text{Ab},\quad P\mapsto\Hom_A(M,[N,P]_\textbf{Z})$$ is representable for every $M\in\text{Mod-}A,N\in A\M$. So, given a commutative ringed topos $(\E,k)$, the category $k\M$ is symmetric monoidal closed (Corollary \ref{monoidal-closed}) under the internal $k$-homs and a \emph{tensor product of $k$-modules}, $$-\otimes_k-.$$ 

When taking $\E=\textbf{Set}$, these correspond to their classical analogues (Examples \ref{classical-module}, \ref{classical-tensor}). 

\begin{definition}[Internally projective/injective object]
    \label{internal-proj}Given a closed abelian category $(\C,[-,-]_\C)$, an object $X\in\C$ is \emph{internally projective} (resp. \emph{internally injective}) if the functor $[X,-]_\C$ (resp. the functor $[-,X]_\C$) is exact. 
\end{definition}

\begin{definition}[Enriched projective/injective object]
    \label{enriched-proj}Given a closed abelian category $\C$, an object $X\in\C$ is \emph{enriched projective} (resp. \emph{enriched injective}) if it is both projective and internally projective (resp. injective and internally injective). 
\end{definition}

When $\E$ is a Grothendieck topos, $k\M$ is an abelian category satisfying axioms (AB3) and (AB5) \cite{tome1}. Therefore, it has enough injectives, and by \cite{harting}, it has enough \emph{internal injectives}. As such, we can develop homological algebra, and give a derived functor definition for Hochschild cohomology in a given topos. 

\subsection{Internal logic of a topos}

Throughout this paper, we reason using the internal (higher-order) logic of a topos, taking \cite{handbook} Chapter 6 as a guide. We use as standard notation $x\!:\!X$ to mean ``$x$ is a variable of type $X$", and $\vDash\phi$ to mean ``this formula is valid in our [given] topos". 

We fix

\begin{itemize}
    \item A commutative ringed topos $(\E,k)$ such that Ab$(\E)$ has a set of cogenerators, e.g., a topos with a natural number object (see page \pageref{rep}).
    \item A $k$-algebra $A$ and $A$-bimodules $M,N$.
\end{itemize}. 
In our logic reasoning, we will make frequent reference to variables
\begin{itemize}
    \item $a_0,...,a_{n+1}\!:\!A$,
    \item $m\!:\!M$,
    \item $f\!:\![A\ot{n},M]_k$.
\end{itemize} We write $m\otimes n$ as the term $\otimes(m,n)\!:\!M\otimes_kN$, where $\otimes$ is the universal morphism $M\times N\xrightarrow{\otimes}M\otimes_kN$. For example, in section 2 we will often reason with the term $a_0\otimes...\otimes a_n$ of type $A\ot{n+1}$. We make frequent use of \cite{handbook} (T53) and (T65), dropping left universal quantifiers in much of our notation.

\section{Internal Hochschild cohomology in a topos}

\subsection{Internal Hochschild cochain complex}

We define a cochain complex \begin{align}0\rightarrow M\xrightarrow{\underline{d}^0}[A,M]_k\xrightarrow{\underline{d}^1}[A^{\otimes2},M]_k\xrightarrow{\underline{d}^2}[A^{\otimes3},M]_k\xrightarrow{\underline{d}^3}...\label{cochain}
\end{align}
as follows. Fix an $n\in\mathbb{N}$. For any $0<i\leq n$, define a $k$-module homomorphism $$\delta_i:A^{\otimes n+1}\rightarrow A^{\otimes n},\quad\delta_i=\id_A\ot{i-1}\otimes\mu_A\otimes\id_A\ot{n-i},$$ for $\mu_A$ the multiplication on $A$. In other words, $\delta^i$ acts as $A$-multiplication on the $i^\text{th}$ and $(i+1)^\text{th}$ factor, and as the identity on all others, so that $$\vDash\delta_i(a_0\otimes...\otimes a_n)=a_0\otimes...\otimes a_{i-1}a_i\otimes...\otimes a_n.$$ Then, define $\underline{\partial}^i=[\delta_i,M]_k:[A^{\otimes n},M]_k\rightarrow[A^{\otimes n+1},M]_k$, so that $$\vDash(\upartial^if)(a_0\otimes...\otimes a_n)=f(a_0\otimes...\otimes a_{i-1}a_i\otimes...\otimes a_n).$$

For the same fixed $n$, we define $\upartial^0:[A^{\otimes n},M]_k\rightarrow[A^{\otimes n+1},M]_k$ as corresponding to the composite $$A\otimes_k[A^{\otimes n},M]_k\otimes_kA^{\otimes n}\xrightarrow{\id\otimes\ev}A\otimes_kM\xrightarrow{l}M,$$ where $l$ is the left module action on $M$, via the adjunction 
\begin{align}
    k\text{-Mod}([A^{\otimes n},M]_k,[A^{\otimes n+1},M]_k)&\cong k\text{-Mod}([A^{\otimes n},M]_k\otimes_kA^{\otimes n+1},M)\notag\\
    &\cong k\text{-Mod}([A^{\otimes n},M]_k\otimes_kA\otimes_kA^{\otimes n},M)\notag\\
    &\cong k\text{-Mod}(A\otimes_k[A^{\otimes n},M]_k\otimes_kA^{\otimes n},M).\notag
\end{align} 
We have $\vDash(\upartial^0f)(a_0\otimes...\otimes a_n)=a_0f(a_1\otimes...\otimes a_n)$, by Lemma \ref{adjunction}.

Similarly, we define $\upartial^{n+1}:[A^{\otimes n},M]_k\rightarrow[A^{\otimes n+1},M]_k$ as corresponding to the composite $$[A^{\otimes n},M]_k\otimes A\ot{n}\otimes A\xrightarrow{\ev\otimes\id}M\otimes_kA\xrightarrow{r}M,$$ where $r$ is the right module action on $M$, via the adjunction
\begin{align}
    k\text{-Mod}([A^{\otimes n},M]_k,[A^{\otimes n+1},M]_k)&\cong k\text{-Mod}([A^{\otimes n},M]_k\otimes A^{\otimes n+1},M)\notag\\
    &\cong k\text{-Mod}([A^{\otimes n},M]_k\otimes A\ot{n}\otimes A,M).\notag
\end{align} 
We have $\vDash(\upartial^{n+1}f)(a_0\otimes...\otimes a_n)=f(a_1\otimes...\otimes a_{n-1})a_n$, by Lemma \ref{adjunction}. 

Then, let $\ud^n:[A\ot{n},M]_k\rightarrow[A\ot{n+1},M]_k$ be the $k$-module morphism $$\ud^n=\sum_{i=0}^{n+1}(-1)^i\upartial^i.$$ We stipulate $A\ot{0}=k$, so that $[A\ot{0},M]_k\cong M$. As a result, $\ud^n$ is defined for all $n\geq0$. We argue that (\ref{cochain}) forms a cochain complex of $k$-modules under these differentials. 

\begin{lemma}
    If $i>j$, then $\upartial^i\circ\upartial^j=\upartial^j\circ\upartial^{i-1}$. 
\end{lemma}

\begin{proof}
By Propositions \ref{handbooktwo-tensor} and \ref{tensor-identities}, and the internal logical evaluation of our differential maps above,  the proof is simple bookkeeping, and requires separate easy verifications for the cases determined by when $i,j$ are adjacent, when $j=0$ or $n+1$, and when $i=n+2$.
\end{proof}

\begin{proposition}
    For all $n\geq0$, $\ud^n$ is a differential on (\ref{cochain}), i.e. $\ud^{n+1}\circ\ud^n=0$.
\end{proposition}

\begin{proof}
    We need to show that $\ud^{n+1}(\ud^n)=0$. We have that $\ud^{n+1}(\ud^n)=\sum_{i=0}^{n+2}(-1)^i\upartial^i(\ud^n)=\sum_{i=0}^{n+2}(-1)^i\upartial^i\left(\sum_{j=0}^{n+1}(-1)^j\upartial^j\right)=\sum_{i=0}^{n+2}\sum_{j=0}^{n+1}(-1)^{i+j}\upartial^i(\upartial^j)$. Thus, we can rewrite (as in \cite{doyle}) as $$\ud^{n+1}(\ud^n)=\sum_{0\leq j<i\leq n+2}(-1)^{i+j}\upartial^i(\upartial^j)+\sum_{0\leq i\leq j\leq n+1}(-1)^{i+j}\upartial^i(\upartial^j).$$ Then by the previous lemma, for $i>j$ $\upartial^i(\upartial^j)=\upartial^j(\upartial^{i-1})$. We obtain $$\ud^{n+1}(\ud^n)=\sum_{0\leq j<i\leq n+2}(-1)^{i+j}\upartial^j(\upartial^{i-1})+\sum_{0\leq i\leq j\leq n+1}(-1)^{i+j}\upartial^i(\upartial^j).$$
    
    As $j\leq i-1$, we can relabel $i-1$ as $i$ and swap indices. We obtain $$\ud^{n+1}(\ud^n)=\sum_{0\leq i\leq j\leq n+1}(-1)^{i+j+1}\upartial^i(\upartial^j)+\sum_{0\leq i\leq j\leq n+1}(-1)^{i+j}\upartial^i(\upartial^j),$$ and the sums cancel. 
\end{proof}

\begin{definition}[Interal Hochschild cohomology]
    We define the \emph{internal Hochschild cochain complex} of $A$ over $k$ with coefficients in $M$, $C[A;M]$, to be the complex of $k$-modules

    $$0\rightarrow M\xrightarrow{\underline{d}^0}[A,M]_k\xrightarrow{\underline{d}^1}[A^{\otimes2},M]_k\xrightarrow{\underline{d}^2}[A^{\otimes3},M]_k\xrightarrow{\underline{d}^3}...$$

    We then define the \emph{internal Hochschild cohomology} of $A$ with coefficients in $M$ as the cohomology of this complex, $$\underline{\HH}^\bullet(A,M)=H^\bullet C[A;M].$$ 
\end{definition}

\begin{example}
    \label{classical}Taking $\E=\textbf{Set}$, we have that $[A\ot{n},M]_k=\Hom_k(A\ot{n},M)$, and we obtain the classical Hochschild cochain complex. As such, internal Hochschild cohomology in \textbf{Set} is precisely the classical Hochschild cohomology of algebras. 
\end{example}

\subsection{Derived functor definition of Hochschild cohomology}\label{ext-section}

Let $\E$ be a Grothendieck topos, so that $k\M$ has enough injectives \cite{tome1}.

\subsubsection*{The Bar Resolution}

We follow the construction of \cite{ltcc} and \cite{doyle}. Given a $k$-algebra $A$, write $A^\text{op}$ for the \emph{opposite algebra} of $A$, i.e. $A=A^\text{op}$ as a $k$-module, but the multiplication $*$ is given by the composite $A^\text{op}\times A^\text{op}\xrightarrow{(\pi_2,\pi_1)}A^\text{op}\times A^\text{op}\xrightarrow{\mu_A}A^\text{op}$, i.e. so that $\vDash a_1*a_2=a_2a_1$. Then, the \emph{enveloping algebra} of $A$ is $$A^e:=A\otimes_kA^\text{op}.$$ Any $A$-bimodule $M$ is a left $A^e$-module via 
$$A^e\otimes_kM\cong A\otimes_kM\otimes_kA\xrightarrow{l_M}M\otimes_kA\xrightarrow{r_M}M,$$ so that for variable $a_1,a_2\!:\!A,m\!:\!M$, we have $\vDash(a_1\otimes a_2)m=a_1ma_2$. 

\begin{definition}
    The \emph{bar complex} is the cochain complex $A^{\otimes \bullet+2}=(A^{\otimes n+2})_{n\geq0},$ with differentials 
    $$d'_n=\sum_{i=0}^n(-1)^i\delta'_i:A^{n+2}\to A^{n+1},$$ where $\delta'_i=\id\ot{i}\otimes\mu\otimes\id\ot{n-i}$, so that for variables $a_0,...,a_{n+1}\!:\!A$, we have $\vDash\delta_i'(a_0\otimes...\otimes a_{n+1})=a_0\otimes...\otimes a_ia_{i+1}\otimes...\otimes a_{n+1}$. 
\end{definition}

\begin{proposition}
    For all $n\geq0$, $d'_n$ is a differential on the bar complex, i.e. $d'_n\circ d'_{n+1}=0$.
\end{proposition}

\begin{proof}
    $d_n'\circ d_{n+1}'=\sum_{i=0}^n\sum_{j=0}^{n+1}(-1)^{i+j}\delta'_i\circ\delta'_j$. Thus, we can rewrite as $$d_n'\circ d_{n+1}'=\sum_{0\leq i<j\leq n+1}(-1)^{i+j}\delta'_i\circ\delta'_j+\sum_{0\leq j\leq i\leq n}(-1)^{i+j}\delta'_i\circ\delta'_j.$$ If $i<j$, then $\delta'_i\circ\delta'_j=\delta'_{j-1}\circ\delta'_i$ (by Propositions \ref{handbooktwo-tensor} and \ref{tensor-identities}, the proof of this is, again, simple bookkeeping for various cases). So, $$d_n'\circ d_{n+1}'=\sum_{0\leq i<j\leq n+1}(-1)^{i+j}\delta'_{j-1}\circ\delta'_i+\sum_{0\leq j\leq i\leq n}(-1)^{i+j}\delta'_i\circ\delta'_j.$$ As $j-1\geq i$, we can relabel $j-1$ as $j$ and swap indices. This becomes $$d_n'\circ d_{n+1}'=\sum_{0\leq j\leq i\leq n}(-1)^{i+j+1}\delta'_i\circ\delta'_j+\sum_{0\leq j\leq i\leq n}(-1)^{i+j}\delta'_i\circ\delta'_j,$$ and the sums cancel. 
\end{proof}

\begin{proposition}
    The complex $A\ot{\bullet+2}$ is a resolution of $A$ as an $A^e$-module.\label{Cbar}
\end{proposition}

\begin{proof} 
We use $A\ot{\bullet+2}\xrightarrow{\mu}A\rightarrow0,$ where $\mu:A\otimes_kA\rightarrow A$ is multiplication on $A$. 


This is exact at $A\otimes_kA$ if $\im\mu\cong\coker d_1'$. We have $\vDash\forall a\!:\!A(\mu(a\otimes1)=a)$, so, by \handbooktwo, $\mu$ is an epimorphism, so $\im\mu\cong A$. The map $\mu:A\otimes_kA\to A$ is zero on the image of $d_1'$, so descends to a map $\mu:A\otimes_kA/\im d_1'\to A$, i.e. $\mu:\coker d_1'\to A$. To show that this is an isomorphism, we construct a linear map $A\to A\times A$ that sends $a\mapsto 1\times a$, and combine with $\otimes$ to get a map $A\to A\otimes_kA$, and show that it is an inverse. 

Immediately, for all $a\!:\!A$, $\vDash\mu(s(a))=\mu(a\otimes1)=a$, so $\mu\circ s=\id_A$. Given $a_0,a_1\!:\!A$, we have that $s(\mu(a_0\otimes a_1))=s(a_0a_1)=1\otimes a_0a_1$. But we have that $\vDash d_1'(1,a_0,a_1)=a_0\otimes a_1-1\otimes a_0a_1$, so we have that in $A\otimes_kA/\im d_1'$ $\vDash a_0\otimes a_1=1\otimes a_0a_1$. So, by Proposition \ref{handbook-tensor}, $s\circ\mu=\id_{\coker d_1'}$. Thus, the complex is exact at $A\otimes_kA$.

We show that there is a contracting homotopy on our cochain complex, and therefore that the rest of the complex is exact (a contracting homotopy implies homotopy equivalence to the trivial complex, and thus zero homology groups \cite{weibel}). We let $s_n$ denote the composite $A\ot{n+1}\cong 1\otimes_kA\ot{n+1}\xrightarrow{1_A}A\ot{n+2}$, so that, for variables $a_0,...,a_{n+1}\!:\!A$, $\vDash s_n(a_0\otimes...\otimes a_{n+1})=1_A\otimes a_0\otimes...\otimes a_{n+1}$. By Propositions \ref{handbooktwo-tensor} and \ref{tensor-identities}, it is simple bookkeeping to show that $d_{n+1}\circ s_n+s_{n-1}\circ d_n=\id_{A^{\otimes n+2}}$. So we have a contracting homotopy on our chain complex, and therefore the complex is exact. \cite{belmans}
\end{proof}

We refer to the bar complex as the \emph{bar resolution} of $A$ as an $A^e$-module. 

\subsubsection*{\texorpdfstring{$\uHH$}{HH} as a derived functor}

\begin{proposition}
    If $A$ is an enriched projective $k$-module (as in Definition \ref{enriched-proj}), $A\ot{\bullet+2}\xrightarrow{\mu}A\rightarrow0$ is a projective resolution of $A^e$-modules. 
\end{proposition}

\begin{proof}
Let $A$ be an enriched projective $k$-module. Then, by Proposition \ref{tensor-projective}, so is $A^{\otimes n}$, for all $n$. We have an $A^e$-module isomorphism $A^{\otimes n+2}\cong A^e\otimes A^{\otimes n}$, as the opposite algebra does not affect the additive group and the action of $A^e$ is only on the first and the last elements.  Therefore, for all $M\in k\M$ we have an isomorphism $\Hom_k(A\ot{n},M)\simeq\Hom_{A^e}(A\ot{n+2},M)$. 
As $A\ot{n}$ is projective, $\Hom_k(A\ot{n},-)$ is exact, and so $\Hom_{A^e}(A\ot{n+2},-)$ is exact. Therefore, $A^{\otimes n+2}$ is projective as an $A^e$-module for all $n$.
\end{proof}

Writing $F$ for the forgetful functor (the restriction of an $A^e$-module to $k$), we have an adjoint pair of functors 

\[\begin{tikzcd}
A^e\text{-Mod}
\arrow[r, "F"{name=F}, bend left=25] &
k\text{-Mod.}
\arrow[l, "-\otimes_k A^e"{name=G}, bend left=25]
\arrow[phantom, from=F, to=G, "\dashv" rotate=90]
\end{tikzcd}\]

We therefore have isomorphism (\cite{tomasic}) \begin{align}
    F[N\otimes_kA^e,M]_{A^e}\cong[N,FM]_k.\label{ext-iso}
\end{align}

\begin{proposition}\label{prop-hh-ext}
\label{ext-prop}If $A$ is an enriched projective $k$-module, we have $$\uHH^n(A,M)\cong F\uExt^n_{A^e}(A,M).$$
\end{proposition}

\begin{proof}
Writing $P_\bullet$ for any projective resolution of $A$ as an $A^e$-mod, 
\begin{align}
    F\uExt^n_{A^e}(A,M)=FH^n[P_\bullet,M]_{A^e}= FH^n[A\ot{\bullet+2},M]_{A^e}&\cong FH^n[A\ot{\bullet}\otimes A^e,M]_{A^e}\notag\\&\cong H^nF[A\ot{\bullet}\otimes A^e,M]_{A^e}\notag\\&\cong H^n[A\ot{\bullet},M]_k=\uHH^n(A,M),\notag
\end{align}
where the second isomorphism is given by the exactness of $F$ and the third by (\ref{ext-iso}).
\end{proof}

We can use this derived functor definition to define the (ordinary) Hochschild cohomology theory. 

\begin{definition}
    We define the \emph{Hochschild cohomology} of $A$ with coefficients in $M$, $$\HH^n(A,M):=\Ext^n_{A^e}(A,M).$$
\end{definition}

\begin{remark}
    This cohomology theory takes values in \textbf{Ab} (ordinary abelian groups in $\textbf{Set}$), rather than in $k\M$ over the ringed topos $(\E,k)$. 
\end{remark}

\begin{example}
    In \textbf{Set}, the internal homs are precisely the hom-sets, and so $\HH^n(A,M)=\uHH^n(A,M)$. 
\end{example}

\section{Internal Hochschild cohomology in difference sets}\label{s-hh-diff}

\subsection{The topos of difference sets}

By a \emph{difference set}, we mean a pair $(X,\sigma_X)$, where $X$ is a set and $\sigma_X$ is an endomorphism on $X$. By a \emph{difference morphism} $(X,\sigma_X)\rightarrow(Y,\sigma_Y)$, we mean a commutative diagram \[\begin{tikzpicture}[scale=2]
	\node (A1) at (0,0.8) {$X$};
	\node (A2) at (0,0) {$X$};
	\node (B1) at (1,0.8) {$Y$};
	\node (B2) at (1,0) {$Y$};
	\draw[->]
	(A1) edge node [above] {$f$} (B1)
	(A1) edge node [left] {$\sigma_X$} (A2)
	(A2) edge node [below] {$f$} (B2)
	(B1) edge node [right] {$\sigma_Y$} (B2);
\end{tikzpicture}\]
Together, these form a category $$\boldsymbol{\sigma}\textbf{-Set}.$$ In practice, we write its objects as $X$, omitting the endomorphism, and write $\lfloor X\rfloor$ as the underlying set.

\begin{example}
    \label{N+}We write $N_+=(\mathbb{N},s)$, where $s:i\mapsto i+1$. 
\end{example}

\begin{example}
    \label{N[s]}We write $\mathbb{N}[\sigma]$ as the difference set of polynomials in $\sigma$ with coefficients in $\mathbb{N}$, where the difference action is multiplication by $\sigma$. Then $$\mathbb{N}[\sigma]\cong\bigoplus_{n\in\mathbb{N}}N_+.$$ 
\end{example}

We can then define difference algebra as algebra internal to this category.

\begin{itemize}
    \item A \emph{(commutative) difference ring} is a pair $(k,\sigma)$, where $k$ is a (commutative) ring and $\sigma:k\rightarrow k$ is a ring endomorphism. 
    \item Given such a $k$, a \emph{difference $k$-module} is a pair $(M,\sigma_M)$, where $M$ is a $\lfloor k\rfloor$-module and $\sigma_M:M\rightarrow M$ an additive endomorphism s.t. $\sigma_M(\lambda m)=\sigma(\lambda)\sigma_M(m)$ for all $\lambda\in k,m\in M$. 
    \item A \emph{difference $k$-algebra} is a difference $k$-module that is also a difference ring. 
\end{itemize}

\begin{remark}
    Let $\boldsymbol{\sigma}$ be the category with single object $o$ such that $\Hom_\sigma(o)$ consists of composites of a single endomorphism $\sigma$ (so $\Hom_\sigma(o)\cong\mathbb{N}$). Then $$\boldsymbol{\sigma}\textbf{-Set}\simeq[\boldsymbol{\sigma}^\text{op},\textbf{Set}],$$ i.e. $\boldsymbol{\sigma}\textbf{-Set}$ is a Grothendieck topos \cite{tomasic}. 
\end{remark} 

So, we can consider topos constructions in the difference case. Crucially, a difference ring is a ring object in $\boldsymbol{\sigma}\textbf{-Set}$, and the difference and topos-theoretic definitions of $k$-module coincide. 

The \emph{internal hom} $[X,Y]$ of difference sets $X$ and $Y$ is obtained by setting the underlying set to be $$\boldsymbol{\sigma}\textbf{-Set}(N_+\times X,Y)=\{(f_j)_{j\in\mathbb{N}}|f_j\in\Hom(\lfloor X\rfloor,\lfloor Y\rfloor),f_{j+1}\circ\sigma_X=\sigma_Y\circ f_j\},$$ where $N_+=(\mathbb{N},i\mapsto i+1)$. This inherits a difference operator from $N_+$, sending $(f_j)_{j\in\mathbb{N}}\mapsto (f_{j+1})_{j\in\mathbb{N}}$. In other words, elements of $[X,Y]$ are commutative ``ladders", 
\[\begin{tikzcd}[column sep=3em, row sep=3em] 
& X\arrow[d,"\sigma_X"] \arrow[r,"f_0"] & Y\arrow[d,"\sigma_Y"] \\
& X\arrow[d,"\sigma_X"] \arrow[r,"f_1"] & Y\arrow[d,"\sigma_Y"] \\
& \vdots & \vdots
\end{tikzcd}\]

so that the difference operator shifts the ladder up by a rung (forgetting the first rung). 

Given a difference ring $k$ and $k$-modules $M,N$, we can construct $[M,N]_k$ as the set of ``ladders" whose rungs are $\lfloor k\rfloor$-module homomorphisms.

\subsubsection{Functors on \texorpdfstring{$\boldsymbol{\sigma}\textbf{-Set}$}{sigma-Set}}

The underlying set $\lfloor X\rfloor$ of a difference set forms the pullback of an essential geometric surjection $\textbf{Set}\to\boldsymbol{\sigma}\textbf{-Set}$ \cite{tomasic}, given by 
\[\begin{tikzpicture}[scale=1]
	\node (B1) at (0.5,0) {\textbf{Set}};
	\node (A2) at (0,1.4) {$\dashv$};
        \node (C2) at (1.2,1.4) {$\dashv$};
	\node (B3) at (0.5,2.8) {$\boldsymbol{\sigma}\textbf{-Set}$};
        \draw[->]
        (B3) edge node [right] {$\lfloor\,\rfloor$} (B1);
	\draw[->,bend left=60]
	(B1) edge node [left] {$\lceil\,|$} (B3);
	
        \draw[->,bend right=60]
        (B1) edge node [right] {$|\,\rceil$} (B3);
\end{tikzpicture}\]
The left adjoint $\lceil-|:\textbf{Set}\to\boldsymbol{\sigma}\textbf{-Set}$ is given by $$\lceil S|=\coprod_{i\in\mathbb{N}}S_i,\quad\sigma:S_i\mapsto S_{i+1},$$ where $S_i\cong S$ for all $i\in\mathbb{N}$. In other words, $$\lceil S|=\{(s,i)|s\in S,i\in\mathbb{N}\},\quad\sigma(s,i)=(s,i+1).$$ The right adjoint $|-\rceil:\textbf{Set}\to\boldsymbol{\sigma}\textbf{-Set}$, i.e. the pushforward of the geometric morphism, is given by $$|S\rceil=\prod_{i\in\mathbb{N}}S_i,\quad\sigma:\prod_{i\geq0}S_i\xrightarrow{\pi_0}\prod_{i\geq1}S_i\xrightarrow{s}\prod_{i\geq0}S_i,$$ where $S_i\cong S$ for all $i\in\mathbb{N}$, and $\pi_0$ is the projection on the first component and $s$ is the shift. 

Therefore, for all $X\in\boldsymbol{\sigma}\textbf{-Set},S\in\textbf{Set}$, we obtain natural isomorphisms $$\boldsymbol{\sigma}\textbf{-Set}(\lceil S|,X)\cong\Hom(S,\lfloor X\rfloor);\quad\quad\Hom(\lfloor X\rfloor,S)\cong\boldsymbol{\sigma}\textbf{-Set}(X,|S\rceil).$$
\begin{example}
    We can write $N_+=\lceil\{*\}|$. From Example \ref{N[s]}, we have $\mathbb{N}[\sigma]\cong\bigoplus_{n\in\mathbb{N}}N_+$. As $\lceil-|$ is a left adjoint, it preserves colimits. So, $$\mathbb{N}[\sigma]\cong\bigoplus_{n\in\mathbb{N}}N_+\cong\bigoplus_{n\in\mathbb{N}}\lceil\{*\}|\cong\lceil\bigoplus_{n\in\mathbb{N}}\{*\}|\cong\lceil\mathbb{N}|.$$
\end{example}

Assigning to a set $X$ the identity $(X,\id_X)$ forms the pullback $I$ of an essential geometric morphism $\boldsymbol{\sigma}\textbf{-Set}$ \cite{tomasic}, given by 
\[\begin{tikzpicture}[scale=1]
	\node (B1) at (0.5,0) {\textbf{Set}};
	\node (A2) at (0,1.4) {$\dashv$};
        \node (C2) at (1.2,1.4) {$\dashv$};
	\node (B3) at (0.5,2.8) {$\boldsymbol{\sigma}\textbf{-Set}$};
        \draw[->]
        (B1) edge node [right] {$I$} (B3);
	\draw[->,bend left=60]
	(B3) edge node [right] {$\Fix$} (B1);
	
        \draw[->,bend right=60]
        (B3) edge node [left] {$\text{Quo}$} (B1);
\end{tikzpicture}\]
The left adjoint $\text{Quo}:\boldsymbol{\sigma}\textbf{-Set}\to\textbf{Set}$ is given by $$\text{Quo}(X)=\text{coeq}(X\stackrel[\id_X]{\sigma_X}{\rightrightarrows}X).$$
The right adjoint $\Fix:\boldsymbol{\sigma}\textbf{-Set}\rightarrow\textbf{Set}$ is given by $$\Fix(X)=\{x:X|\sigma_X(x)=x\}.$$ This naturally specialises to a functor $k\M\rightarrow\Fix(k)\M$, which we will also write as $\Fix$ by a slight abuse of notation. In particular, we note that$$\Fix([X,Y])=\boldsymbol{\sigma}\textbf{-Set}(X,Y),\quad\Fix([M,N]_k)=\Hom_k(M,N),$$ for all difference sets $X$ and $Y$, and for all $k$-modules $M$ and $N$ over a given difference ring $k$. 

The functor $\text{Quo}$ restricted to difference abelian groups is called the \emph{coinvariants} functor 
$$(-)_\sigma:\boldsymbol{\sigma}\textbf{-Ab}\rightarrow\textbf{Ab},\quad(X,\sigma_X)\mapsto X_\sigma:=X/\im(\sigma_X-\id_X).$$ Again, this naturally specialises to a functor $k\M\rightarrow\Fix(k)\M$, which we will also write as $(-)_\sigma$ by a slight abuse of notation. 

The functor $\text{Fix}$ restricted to difference abelian groups is sometimes called the \emph{invariants} functor and denoted
$$
(-)^\sigma:\boldsymbol{\sigma}\textbf{-Ab}\rightarrow\textbf{Ab}.
$$

Then, we have that $R^1\Fix=(-)_\sigma$, and $R^i\Fix=0$ for all $i>1$ \cite{tomasic}.

\subsubsection{Tensored structure of \texorpdfstring{$k\M$}{k-Mod} over \texorpdfstring{$\boldsymbol{\sigma}\textbf{-Set}$}{sigma-Set}}

Let $k$ be a commutative difference ring. As shown in \cite[15.4]{tomasic}, the category $k\M$ is tensored over difference sets. For a difference set $E$ and a difference $k$-module $M$, we write
$$
E\otimes M=\bigoplus_{e\in E}M, \text{ where } \sigma(e,m)=(\sigma_E(e),\sigma_M(m)). 
$$
It has the universal property
$$
[E\otimes M,N]_k\simeq [E,[M,N]_k].
$$


    


\begin{lemma}
    We have that $N_+$ is internally projective in \emph{$\boldsymbol{\sigma}\textbf{-Set}$}. 
\end{lemma}

\begin{proof}
    Let $(f_i)_{i\in\mathbb{N}}\in[N_+,Z]$. We have that $f_{i+1}(j+1)=\sigma_Z(f_i(j))$ $\forall i,j\in\mathbb{N}$, so $(f_i)_{i\in\mathbb{N}}$ is determined by the values $f_0(j)$ for all $j$ and $f_i(0)$ for all $i$. So, we have an isomorphism $\lfloor[N_+,Y]\rfloor\cong\lfloor Y\rfloor^\mathbb{Z}\cong\Hom(\mathbb{Z},\lfloor Y\rfloor)$, identifying $(f_i)_{i\in\mathbb{N}}$ with $(x_i)_{i\in\mathbb{Z}}$, where $$x_i=\begin{cases}
	f_0(i) & \text{if} \: i\geq0\\
	f_i(0) & \text{if} \: i<0\end{cases}$$
    Let $q:Y\epi Z$ be a surjection in $\boldsymbol{\sigma}\textbf{-Set}$. This induces difference map $q^*:[N_+,Y]\to[N_+,Z]$, i.e. $q^*:\Hom(\mathbb{Z},\lfloor Y\rfloor)\to\Hom(\mathbb{Z},\lfloor Z\rfloor)$, where the bare hom-sets inherit difference structure via the isomorphism. By the Axiom of Choice in \textbf{Set}, this is a surjection. As a surjective difference map, $q^*$ is a surjection in $\boldsymbol{\sigma}\textbf{-Set}$. 
\end{proof}

\begin{lemma}
    \label{proj-E}Let $X$ be a set. Then $E=\lceil X|$ is projective in \emph{$\boldsymbol{\sigma}\textbf{-Set}$}. 
\end{lemma}

\begin{proof}
    We have that $\boldsymbol{\sigma}\textbf{-Set}(E,Y)\cong\Hom(X,\lfloor Y\rfloor)$. Let $q:Y\epi Z$ be a surjection in $\boldsymbol{\sigma}\textbf{-Set}$. Then it induces a map $q^*:\boldsymbol{\sigma}\textbf{-Set}(E,Y)\to\boldsymbol{\sigma}\textbf{-Set}(E,Z)$, i.e. $q^*:\Hom(X,\lfloor Y\rfloor)\to\Hom(X,\lfloor Z\rfloor)$. By the Axiom of Choice in \textbf{Set}, this is a surjection. 
\end{proof}

\begin{proposition}
    \label{tensored-internally-projective}Let $X$ be a set, and write $E=\lceil X|$. Then $E\otimes k$ is internally projective in $k\M$.
\end{proposition}

\begin{proof}
    For any $Z\in k\M$, we have natural $k$-module isomorphisms
    \begin{align}
        [E\otimes k,Z]_k\cong[E,[k,Z]_k]\cong[E,Z]\cong(\boldsymbol{\sigma}\textbf{-Set}(N_+\times E,Z),\sigma)\cong(\boldsymbol{\sigma}\textbf{-Set}(E,[N_+,Z]),\sigma),\notag
    \end{align}
    where each object obtains a $k$-module structure via the underlying difference bijections, and the difference structure on the right is obtained via the shift on $N_+$. Let $q:Y\epi Z$ be a surjection in $k\M$. By the internal projectivity of $N_+$ in $\boldsymbol{\sigma}\textbf{-Set}$, we obtain surjection $q^*:[N_+,Y]\to[N_+,Z]$. By the projectivity of $E$ in $\boldsymbol{\sigma}\textbf{-Set}$, we obtain surjection $q^*:\boldsymbol{\sigma}\textbf{-Set}(E,[N_+,Y])\to\boldsymbol{\sigma}\textbf{-Set}(E,[N_+,Z])$, i.e. a surjection $q^*:[E\otimes k,Y]_k\to[E\otimes k,Z]_k$. 
\end{proof}

\begin{proposition}
    \label{tensored-projective}Let $X$ be a set, and write $E=\lceil X|$. Then $E\otimes k$ is projective in $k\M$.
\end{proposition}

\begin{proof}
    For any $Z\in k\M$, we have natural isomorphisms of abelian groups $$\Hom_k(E\otimes k,Z)\cong\Hom_k(k\otimes_k(E\otimes k),Z)\cong\Hom_k(k,[E\otimes k,Z]_k).$$

    Let $q:Y\epi Z$ be a surjection in $k\M$. By Proposition \ref{tensored-internally-projective}, we have a $k$-module surjection $q^*:[E\otimes k,Y]_k\to[E\otimes k,Z]_k$, and so, by the projectivity of $k$ as a $k$-module, we have a surjection of abelian groups $q^*:\Hom_k(k,[E\otimes k,Y]_k)\to\Hom_k(k,[E\otimes k,Z]_k)$, i.e. a surjection $q^*:\Hom_k(E\otimes k,Y)\to\Hom_k(E\otimes k,Z)$. 
\end{proof}

\begin{corollary}
    \label{tensored-enriched-proj}Let $X$ be a set, and write $E=\lceil X|$. Then $E\otimes k$ is enriched projective in $k\M$.
\end{corollary}

\subsection{Low dimensional internal difference Hochschild cohomology}

\subsubsection{Degree 0} 

We have an isomorphism $[k,M]_k\cong M$, where every $m\in M$ can be seen as a ladder $(\sigma^j(m))_{j\in\mathbb{N}}$ such that $\sigma^j(m)(\lambda)=\lambda\sigma^j(m)$, for all $j$. Then $\ud^1:M\cong[k,M]_k\rightarrow[A,M]_k$, $\ud^1m_j(a)=a\sigma^j(m)-\sigma^j(m)a$. 

So, $$\uHH^0(A,M)=\ker\ud^1=\{m\in M|\:a\sigma^i(m)=\sigma^i(m)a\:\forall i\in\mathbb{N},a\in A\}.$$

\subsubsection{Degree 1}

Similarly, we can explicitly consider $\uHH^1(A,M)$ by looking at $\ud^2:[A,M]_k\rightarrow [A\otimes_kA,M]_k$, where $(\ud^2(f_j))_j(a\otimes b)=af_j(b)-f_j(ab)+f_j(a)b$, as expected. Then 
\begin{align}
\ker\ud^2
&=\{(f_j)_{j\in\mathbb{N}}\in[A,M]_k|\forall a,b\in A,f_j(ab)=af_j(b)+f_j(a)b\}.\notag
\end{align}
Call these \emph{internal difference $k$-derivations}, and write $\underline{\Der}_k(A,M):=\ker\ud^2$. To get the first internal Hochschild cohomology we quotient out by $$\im{\ud^1}=\{(f_j)_{j\in\mathbb{N}}\in[A,M]_k|\:\exists m\in M\text{ s.t. }\forall a\in A,f_j(a)=a\sigma_M^j(m)-\sigma_M^j(m)a\}.$$ Call these \emph{internal difference inner $k$-derivations}, and write $\underline{\text{IDer}}_k(A,M):=\im\ud^1$. So, we have canonical isomorphism $$\uHH^1(A,M)\cong\underline{\Der}_k(A,M)/\underline{\text{IDer}}_k(A,M).$$ 

\subsection{Example}

Let $k$ be a commutative difference ring, and let 
$$
k\{x\}=k[x_0,x_1,\ldots], \text{ with } \sigma:x_i\mapsto x_{i+1}
$$
be the difference polynomial ring. Write $p(x)$ as shorthand for a polynomial $p(x_1,...,x_n)\in k\{x\}$. 

\begin{lemma}
    \label{k{x}=N}We have an isomorphism of $k$-modules $$k\{x\}\cong\mathbb{N}[\sigma]\otimes k.$$ 
\end{lemma}

\begin{proof}
    We identify $\nu=n_0+n_1\sigma+n_2\sigma^2+...+n_p\sigma^p\in\mathbb{N}[\sigma]$ with $x^\nu=x_0^{n_0}x_1^{n_1}x_2^{n_2}...x_p^{n_p}\in k\{x\}$. Then $\sigma(x^\nu)=x_1^{n_0}x_2^{n_1}...x_{p+1}^{n_p}=x^{\sigma(\nu)}$. Then $k\{x\}=\bigoplus_{\nu\in\mathbb{N}[\sigma]}k$, where $\sigma(x^\nu,\lambda)=(x^{\sigma(\nu)},\sigma_k(\lambda))$, i.e. $k\{x\}\cong\mathbb{N}[\sigma]\otimes k$. 
\end{proof}




\begin{proposition}
    We have that $k\{x\}$ is an enriched projective $k$-module.
\end{proposition}

\begin{proof}
    From Lemma \ref{k{x}=N} and Example \ref{N[s]}, we know that $k\{x\}\cong\mathbb{N}[\sigma]\otimes k\cong\lceil\mathbb{N}|\otimes k$. By Corollary \ref{tensored-enriched-proj}, $k\{x\}$ is enriched projective. 
\end{proof}

So, by the workings of Proposition \ref{ext-prop}, we calculate the internal Hochschild cohomology of $k\{t\}$ by constructing a projective resolution of $k\{t\}$ as a $k\{t\}^e$ module. 

\begin{lemma}
    $k\{t\}^e\cong k\{x,y\}$.
\end{lemma}

\begin{proof}
    Given $k\{t\}$ is a commutative difference ring, we know that $k\{t\}^e=k\{t\}\otimes_kk\{t\}$. Consider a map $k\{t\}^e\rightarrow k\{x,y\},p(t)\otimes q(t)\mapsto p(x)q(y)$. This map is an injective $k$-algebra homomorphism. Likewise, given $f(x,y)\in k\{x,y\}$, it can be factorised as a linear sum of terms $p(x)q(y)$, and we have surjectivity. 
\end{proof}

\subsubsection{Projective resolution of \texorpdfstring{$k\{x\}$}{k{x}}}

There exists a projective resolution 
\begin{align}
\cdots\to\Pi\xrightarrow{h}N_{+}^2\otimes k\{x,y\}\xrightarrow{g}N_{+}\otimes k\{x,y\}\xrightarrow{f}k\{x,y\}\xrightarrow{\epsilon}k\{x\}\to 0\label{projres}
\end{align} of $k\{x\}$ as a $k\{x,y\}$-module, where $\Pi:=(N_+^2\otimes k\{x,y\})\oplus(N_+^3\otimes k\{x,y\})$, also projective as the direct sum of projective modules.

We will use this to calculate the cohomology groups of $k\{x\}$ in section \ref{kt}, but we must first prove that (\ref{projres}) is a projective resolution of $k\{x\}$ as a $k\{x,y\}$-module. Write $z_i=x_i-y_i\in k\{x,y\}$. 

Define the $k\{x,y\}$-module map $f:N_+\otimes k\{x,y\}\to k\{x,y\},(i,p)\mapsto z_ip$. It is not hard to see that the image of $f$ is precisely the kernel of $\epsilon:x,y\mapsto x$. We have $$\ker f=\{\sum_i(i,p_i)\in N_+\otimes k\{x,y\}|\sum_iz_ip_i=0\}.$$ 

Define the $k\{x,y\}$-module map $g:N_+\times N_+\otimes k\{x,y\}\to N_+\otimes k\{x,y\},(i,j,p)\mapsto(i,z_jp)-(j,z_ip)$. So, \begin{align}
    g:\sum_{i,j}(i,j,p_{ij})\mapsto\sum_{i}(i,\sum_jz_j(p_{ij}-p_{ji})).\label{g}
\end{align} 

\begin{proposition}
We have $\ker(f)=\im(g).$\label{fg}
\end{proposition}

\begin{proof}
For convenience, via a linear variable change, we may consider $k\{x,y\}=k\{z,y\}$ as $R\{z\}$, for $R=k\{y\}$. 

Suppose, for contradiction, a minimal element $\sum_i(i,p_i)\in\ker f$ which is not in $\im g$. So $f(\sum_i(i,p_i))=\sum_i z_ip_i=0$ for some $p_i\in R\{x\}$. 
    
    Firstly, suppose $p_0=0$. For any $i$, write $p_i=p_i^{(0)}+z_0p_i^{(1)}+z_0^2p_i^{(2)}+...$, where $z_0\nmid p_i^{(n)}(z_1,z_2,...)$ for any $n$. Then as $\sum_iz_ip_i=\sum_iz_ip_i^{(0)}+\sum_iz_iz_0p_i^{(1)}+\sum_iz_iz_0^2p_i^{(2)}+...=0$, each of these terms must equal zero (they cannot cancel, as they divide by different powers of $z_0$). So, $\sum_iz_ip_i^{(0)}=\sum_iz_ip_i^{(1)}=...=0$. So, rewriting $p_i'=p_i^{(0)}+p_i^{(1)}+...$, we have that $\sum_iz_ip_i'=0$, so we have an equally minimal example $\sum_i(i,p_i')$. Then, reindex via $(i,p_i'(z_1,z_2,...))\mapsto(i-1,p_i'(z_0,z_1,...))$. We have that $\sum_iz_{i-1}p_i'(z_0,z_1,...)=\sum_iz_ip_i'(z_1,z_2,...)=0$, and so we have an equally minimal example where $p_0\neq0$. So we may proceed with the assumption that our example $\sum_i(i,p_i)$ has $p_0\neq0$. 
    
    By substituting $z_j=0$ for $j\neq i$, we see that $p_i(0,\ldots,0,z_i,0,\ldots 0)=0$, so every term of $p_i$ must contain a factor of $z_j$ for some $j\neq i$, and we may write
    $$
    p_i=\sum_{j\neq i}z_j p_{ij}.
    $$
    By substituting $z_k=0$ for $k\notin \{i,j\}$, we see that $f_{ij}(0,\ldots,z_i,0,\ldots, z_j,\ldots,0)=0$, so we may write 
    $$
    p_{ij}+p_{ji}=\sum_{k\notin\{i,j\}}z_k p_{\{i,j\}k},
    $$
    where the index $\{i,j\}k$ reflects the symmetry in $i,j$.

    With this notation, we obtain
\begin{align}
    \sum_i(i,p_i)=\sum_{i\neq j}(i,z_jp_{ij})&=\sum_{i<j}(i,z_j p_{ij})+\sum_{i<j}(j,z_i p_{ji})\notag\\
    &=\sum_{i<j}(i,z_j p_{ij})+\sum_{i<j}(j,z_i(\sum_{k\notin\{i,j\}}z_k p_{\{i,j\}k}-p_{ij}))\notag\\
    &
    =\sum_{i<j}\left( (i,z_jp_{ij})-(j,z_ip_{ij})\right)+\sum_{i<j, k\notin\{i,j\}}(j,z_iz_k p_{\{i,j\}k})\notag\\
    &
    =\sum_{i<j}g(i,j,p_{ij})+\sum_{i<j, k\notin\{i,j\}}(j,z_iz_k p_{\{i,j\}k}).\notag
\end{align}
The first term lies in $\im(g)$, while the second term is in $\ker(f)$ again. However, the second term does not contain a term $(0,p_0)$. As this was nonzero by assumption, the second term has fewer terms than the expression we started with, and we obtain our contradiction.
\end{proof}

\subsubsection{Internal Hochschild cohomology of \texorpdfstring{$k\{x\}$}{kx}}\label{kt}

We have that $k\{x\}$ is projective as a $k$-module. Therefore, by Proposition \ref{ext-prop}, $$\uHH^n(k\{t\})\cong FH^n[P_\bullet,k\{t\}]_{k\{x,y\}},$$ where $P_\bullet$ here refers to the truncated complex $P_\bullet\to0$, for projective resolution $P_\bullet\to k\{t\}$ of $k\{t\}$ as an $k\{x,y\}$-module. Applying $[-,k\{t\}]_{k\{x,y\}}$ to $P_\bullet\to0$, we obtain complex \begin{align}
    0\to&[k\{x,y\},k\{t\}]_{k\{x,y\}}\xrightarrow{f^*}[N_+\otimes k\{x,y\},k\{t\}]_{k\{x,y\}}\notag\\&\xrightarrow{g^*}[N_+^2\otimes k\{x,y\},k\{t\}]_{k\{x,y\}}\rightarrow...\notag
\end{align}

Now, $[k\{x,y\},k\{t\}]_{k\{x,y\}}\cong k\{t\}$. Similarly, for any difference set $E$, $$[E\otimes k\{x,y\},k\{t\}]_{k\{x,y\}}\cong[E,[k\{x,y\},k\{t\}]_{k\{x,y\}}]\cong[E,k\{t\}].$$ So, we obtain complex $$0\to k\{t\}\xrightarrow{f^*}[N_+,k\{t\}]\xrightarrow{g^*}[N_+^2,k\{t\}]\rightarrow...$$

\begin{remark}
    Recall that $A\in A^e\M$ via $(a\otimes b)c=acb$. Given $f(x,y)\in k\{x,y\}$, we can write it as a linear sum $\sum_ip_i(x)q_i(y)$. Then, $k\{t\}\in k\{x,y\}\M$ via $f(x,y)g(t)=(\sum_ip_i(x)q_i(y))g(t)=\sum_ip_i(t)g(t)q_i(t)=\sum_ip_i(t)q_i(t)g(t)$. So, the action is $f(x,y)*g(t)=f(t,t)g(t)$. 
\end{remark}

\begin{proposition}
    \label{HH0}We have an isomorphism of $k$-modules $$\uHH^0(k\{t\})\cong k\{t\}.$$
\end{proposition}

\begin{proof}
    We have that $\uHH^0(k\{t\})=\ker f^*$. Now, $f^*:(\alpha_n)_{n\in\mathbb{N}}\mapsto(\alpha_n\circ f)_{n\in\mathbb{N}}$. But for any $\sum_i(i,p_i)$, we have that for all $n$, $\alpha_n(f(\sum_i(i,p_i)))=\alpha_n(\sum_iz_ip_i))=\sum_iz_i\alpha_i(p_i)$, because $\alpha_n$ is a $k\{x,y\}$-module homomorphism, and by the actions of $k\{x,y\}$ on $k\{x,y\}$ and $k\{t\}$. But in $k\{t\}$, $$z_i\alpha_n(p_i)=(x_i-y_i)*\alpha_n(p_i)=(t_i-t_i)\alpha_n(p_i)=0.$$ So, $\uHH^0(k\{t\})=\ker f^*=k\{t\}$. 
\end{proof}

\begin{proposition}
    We have an isomorphism of $k$-modules $$\uHH^1(k\{t\})\cong[N_+,k\{t\}].$$
\end{proposition}

\begin{proof}
    We have that $\uHH^1(k\{t\})=\ker g^*/\im f^*=\ker g^*$, by the proof of Proposition \ref{HH0}. Now, $g^*:(\alpha_n)_{n\in\mathbb{N}}\mapsto(\alpha_n\circ g)_{n\in\mathbb{N}}$. But for any $\sum_{i,j}(i,j,p_{ij})$, we have that for all $n$, 
    \begin{align}
        \alpha_n(f(\sum_{i,j}(i,j,p_{ij})))=\alpha_n(\sum_i(i,z_j(p_{ij}-p_{ji})))&=\sum_i\alpha_n(i,z_j(p_{ij}-p_{ji}))\notag\\&=\sum_i\alpha_n(z_j*(i,(p_{ij}-p_{ji})))\notag\\&=\sum_iz_j\alpha_n(i,(p_{ij}-p_{ji}))=0,\notag
    \end{align} because $\alpha_n$ is a $k\{x,y\}$-module homomorphism, and by the actions of $k\{x,y\}$ on $N_+\otimes k\{x,y\}$ and $k\{t\}$. So, $\uHH^1(k\{t\})=\ker g^*=[N_+\otimes k\{x,y\},k\{t\}]_{k\{x,y\}}\cong[N_+,k\{t\}]$. 
\end{proof}

As $k\{x\}$ is commutative, we know that $\underline{\IDer}_k(k\{x\})=0$, so we have  $$\underline{\Der}_k(k\{x\})\cong[N_+,k\{x\}].$$ 

\subsubsection{Difference derivations on \texorpdfstring{$k\{x\}$}{k{x}}}

We can make explicit the identification $\underline{\Der}_k(k\{x\})\cong[N_+,k\{x\}]$. Suppose $d=(d_i)_{i\in\mathbb{N}}$ is an internal difference derivation on $k\{x\}$. Write $d_i=\sum_jp_i^j\frac{\partial}{\partial x_j}$. Then $d$ corresponds to $(f_i)_{i\in\mathbb{N}}$, where $f_0(i)=p_0^i$ and $f_i(0)=p_i^0$, for all $i\in N_+$. 

Classically, given any $\{p_i|i\in\mathbb{N}\}\in\lfloor k\rfloor[x_1,x_2,...]$, you can define a derivation $d$ on $\lfloor k\rfloor[x_1,x_2,...]$ by setting $d(1)=0,d(x_i)=p_i$. Then, $d=\sum_ip_i\frac{\partial}{\partial x_i}$, where the action of $p$ here is given by $(pd)(f)=pd(f)$. So, we recover the classical result that $$\Der_{\lfloor k\rfloor}(\lfloor k\rfloor[x_1,x_2,...])\cong\bigoplus_{i\in\mathbb{N}}\lfloor k\rfloor[x_1,x_2,...]\frac{\partial}{\partial x_i}.$$ 

Similarly, we have that $\Der_k(k\{x\})\cong\Fix[N_+,k\{x\}]=\boldsymbol{\sigma}\textbf{-Set}(N_+,k\{x\})\cong k\{x\}$. We can see this explicitly. A $\lfloor k\rfloor[x_1,x_2,...]$-derivation $d$ is a difference derivation if $\sigma(d(p(x_1,x_2,...)))=d(\sigma(p(x_1,x_2,...)))$. But, given a term $f=x_1^{i_1}x_2^{i_2}...x_n^{i_n}\in k\{x\}$, we have \begin{align}
    \sigma(d(f))=\sigma(\sum_jp_j\frac{\partial}{\partial x_j}x_1^{i_1}...x_n^{i_n})&=\sigma(\sum_jp_ji_jx_1^{i_1}...x_j^{i_j-1}...x_n^{i_n})\notag\\&=\sum_ji_j\sigma(p_j)x_2^{i_1}...x_{j+1}^{i_j-1}...x_{n+1}^{i_n},\notag
\end{align} whereas 
\begin{align}
    d(\sigma(f))=\sum_jp_j\frac{\partial}{\partial x_j}x_2^{i_1}...x_{n+1}^{i_n}&=\sum_ji_jp_{j+1}x_2^{i_1}...x_{j+1}^{i_j-1}...x_{n+1}^{i_n}.\notag
\end{align} So, $\sigma(d(f))=d(\sigma(f))$ if, for all $i$, $p_{i+1}=\sigma(p_i)$. As a result, a difference derivation on $k\{x\}$ is determined by a choice of $p_1\in k\{x\}$, and we see that $\Der_k(k\{x\})\cong k\{x\}$. 



\subsection{Grothendieck spectral sequence}\label{g-s-s}

In the case that $A\in k\M$ is enriched projective, our internal Hochschild cohomology is determined by the right derived functors of $F\Hom_{A^e}(A,M)$, where $F:A^e\M\rightarrow k\M$ is the forgetful functor. We have that $$F\Hom_{A^e}(A,-)=\Fix\circ F[A,-]_{A^e}.$$ 

\begin{lemma}
    Let $A\in k\text{-Alg}$ be enriched projective as a $k$-module. The functor $F[A,-]_{A^e}$ takes injective objects to $\Fix$-acyclic objects (in particular, to injective objects). 
\end{lemma}

\begin{proof}
    Let $A$ be enriched projective as a $k$-module. Given $M\in A^e\M$ injective, we want to show that $F[A,M]_{A^e}$ is injective in $k\M$. So we want to know if $\Hom_k(-,F[A,M]_{A^e})$ is exact. Let $0\rightarrow X\rightarrow Y\rightarrow Z\rightarrow0$ be a s.e.s. in $A^e$-Mod. $\Hom_k(-,F[A,M]_{A^e})$ is left exact, so we have exact sequence 
$$0\rightarrow\Hom_k(Z,F[A,M]_{A^e})\rightarrow\Hom_k(Y,F[A,M]_{A^e})\rightarrow\Hom_k(X,F[A,M]_{A^e}).$$ 
Now, $\Hom_k(Z,F[A,M]_{A^e})\cong\Hom_{A^e}(Z\otimes_kA^e,[A,M]_{A^e})\cong\Hom_{A^e}((Z\otimes_kA^e)\otimes_{A^e}A,M)$. So we have exact sequence $$0\rightarrow\Hom_{A^e}((Z\otimes_kA^e)\otimes_{A^e}A,M)\rightarrow\Hom_{A^e}((Y\otimes_kA^e)\otimes_{A^e}A,M)\rightarrow\Hom_{A^e}((X\otimes_kA^e)\otimes_{A^e}A,M).$$ Now, by assumption that $M\in A^e\M$ is injective, $\Hom_{A^e}(-,M)$ is exact. So we can consider the exact sequence 
$$(X\otimes_kA^e)\otimes_{A^e}A\rightarrow(Y\otimes_kA^e)\otimes_{A^e}A\rightarrow(Z\otimes_kA^e)\otimes_{A^e}A\rightarrow0.$$ 
Given that $A^e\in A^e\M$, we have isomorphism $(X\otimes_kA^e)\otimes_{A^e}A\cong X\otimes_k(A^e\otimes_{A^e}A)$. We also have $X\otimes_k(A^e\otimes_{A^e}A)\cong X\otimes_kA$. Also, as $A\in k\M$ is projective, $-\otimes_kA$ is exact. So, from our original s.e.s. we obtain a s.e.s. $$0\rightarrow(X\otimes_kA^e)\otimes_{A^e}A\rightarrow(Y\otimes_kA^e)\otimes_{A^e}A\rightarrow(Z\otimes_kA^e)\otimes_{A^e}A\rightarrow0.$$ 
Thus, we have s.e.s. $$0\rightarrow\Hom_k(Z,F[A,M]_{A^e})\rightarrow\Hom_k(Y,F[A,M]_{A^e})\rightarrow\Hom_k(X,F[A,M]_{A^e})\rightarrow0;$$ in other words, $\Hom_k(-,F[A,M]_{A^e})$ is exact. So $F[A,M]_{A^e}\in k\M$ is injective.
\end{proof}

The functors $F[A,-]_{A^e}:A^e\M\rightarrow k\M$ and $\Fix:k\M\rightarrow\Fix(k)\M$ are additive, left exact and between abelian categories, such that both $A^e\M$ and 
$k\M$ have enough injectives. As $F[A,-]_{A^e}$ takes injective objects to $\Fix$-acyclic objects, then for each object $M$ of $A^e\M$ there is a spectral sequence (the \emph{Grothendieck spectral sequence}
\cite{grothendieck}): $$E_{2}^{p,q}=(R^{p}\Fix\circ R^{q}F[A,-]_{A^e})(M)\Longrightarrow R^{p+q}(F\Hom_{A^e}(A,-))(M).$$
Recall that $R^0\Fix=\Fix$; $R^1\Fix=(-)_\sigma$, the coinvariants of the module, and $R^i\Fix=0$ for $i>1$ \cite{tomasic}. We also have that $R^iF[A,-]_{A^e}=F\uExt^i_{A^e}(A,-)$ for all $i$. So in the above case, we have 
\begin{alignat*}{3}
  &E_2^{0,0}=\Fix(F[A,-]_{A^e}) &&E_2^{1,0}=(F[A,-]_{A^e})_\sigma&& \\
  &E_2^{0,q}=\Fix(F\uExt^q_{A^e}(A,-))\qquad&&E_2^{1,q}=F\uExt^q_{A^e}(A,-)_\sigma\qquad&&E_2^{p,q}=0\text{ for all }p>1
\end{alignat*} 

As a result, if $A\in k\text{-Alg}$ is enriched projective as a $k$-module, we have spectral sequence in $\Fix(k)\M$ $$E_{2}^{p,q}=(R^{p}\Fix\circ \uExt^{q}_{A^e}(A,-))(M)\Longrightarrow \Ext^{p+q}_{A^e}(A,-))(M).$$ 

\begin{theorem}\label{th-groth-ss}
\label{ses}If $A\in k\emph{-Alg}$ is enriched projective as a $k$-module, then for all $n>0$ we have short exact sequences in $\Fix(k)\M$ $$0\to\uHH^{n-1}(A,M)_{\sigma}\to\HH^n_{A^e}(A,M)\to\uHH^n(A,M)^\sigma\to0.$$
\end{theorem}

\begin{proof}
    This is immediate from \cite[Exercise 5.2.1]{weibel}, and from Proposition \ref{ext-iso}. 
\end{proof}

\subsection{Internal hypercohomology}

We can compute spectral sequences in our cohomology using the hypercohomology of $\Fix$, applied to our Hochschild cochain complex $[A\ot{\bullet+2},M]_{A^e}\cong[A\ot{\bullet},M]_k$. By construction (2.4.2) and onwards in \cite{grothendieck}, we have second-degree sheets of two spectral sequences, 
\begin{align}
{}^IF^{p,q}_2([A\ot{\bullet+2},M]_{A^e})&=H^pR^q\Fix([A\ot{\bullet+2},M]_{A^e});\notag\\{}^{II}F^{p,q}_2([A\ot{\bullet+2},M]_{A^e})&=R^p\Fix(H^q[A\ot{\bullet+2},M]_{A^e})=R^p\Fix\uHH^q(A,M),\notag
\end{align}
with the last equality given when $A$ is enriched projective in $k\M$. Given such an $A$, both of these spectral sequences converge to the hypercohomology 
\begin{align}
    \mathscr{R}^nF\Fix([A\ot{\bullet+2},M]_{A^e})=R^n(\Fix\circ H^0)([A\ot{\bullet+2},M]_{A^e})=&R^n(\Fix\circ H^0\circ[A\ot{\bullet+2},-]_{A^e})(M)\notag\\
    =&R^n(\Fix\circ\uExt_{A^e}^0(A,-))(M)\notag\\
    =&R^n(\Fix\circ[A,-]_{A^e})(M)\notag\\
    =&R^n(\Hom_{A^e}(A,-))(M)\notag\\
    =&\Ext^n_{A^e}(A,M),\notag
\end{align}
where the enriched projectivity of $A$ gives the exactness of $[A\ot{\bullet+2},-]_{A^e}$, as in Subsection \ref{ext-section}. 

Given our internal Hochschild cochain complex $[A\ot{\bullet},M]_k$, we write $$\uHH^{n,\sigma}(A,M)$$ for the cohomology of $\Fix$ of the complex, and $$\uHH^n_{\sigma}(A,M)$$ for the cohomology of the coinvariants of the complex. 

Given that $R^i\Fix=0$ for all $i>1$, ${}^IF^{p,q}_2$ is a two-row spectral sequence, i.e., $${}^IF^{p,q}_2([A\ot{\bullet+2},M]_{A^e})=0$$ for all $q>1$. 

\begin{theorem}\label{th-long-exseq}
    If $A\in k\emph{-Alg}$ is an enriched projective $k$-module, we obtain a long exact sequence 
    $$\cdots\to\HH^n(A,M)\to\uHH^{n-1}_\sigma(A,M)\to\uHH^{n+1,\sigma}(A,M)\to\HH^{n+1}(A,M)\to\cdots$$
\end{theorem}

\begin{proof}
    This is immediate from \cite[Exercise 5.2.2]{weibel}, and from Proposition \ref{ext-iso}. 
\end{proof}

\begin{proposition}\label{prop-5-term}
    If $A\in k\emph{-Alg}$ is an enriched projective $k$-module, then the 5-term exact sequence of low degree terms \cite{weibel} associated to ${}^{I}F^{p,q}_2$ is
\begin{align}
    0\to\uHH^{1,\sigma}(A,M)\to\HH^1(A,M)&\to\uHH^0_\sigma(A,M)\to\uHH^{2,\sigma}(A,M)\to\HH^2(A,M).\notag
\end{align}
\end{proposition}

\begin{remark}
Proposition~\ref{prop-5-term} merely yields the initial part of the long exact sequence from Theorem~\ref{th-long-exseq}.
\end{remark}

\begin{remark}
    That $R^i\Fix=0$ for all $i>1$ also tells us that ${}^{II}F^{p,q}_2$ is a two-column spectral sequence, so we obtain short exact sequences
    $$0\rightarrow{}^{II}F^{1,n-1}_2([A\ot{\bullet+2},M]_{A^e})\rightarrow\Ext^n_{A^e}(A,M)\rightarrow{}^{II}F^{0,n}_2([A\ot{\bullet+2},M]_{A^e})\rightarrow0.$$
    We can see that this recovers the short exact sequences of the Grothendieck spectral sequence from Theorem $\ref{ses}$. 
\end{remark}

\begin{remark}
    We can make some explicit sense of these terms. The second and fourth terms are, in a sense, the naive Hochschild cohomology groups. For example, $\Fix(\uHH^1(A,M))=H^1\Hom_{A^e}(A\ot{\bullet+2},M)\cong H^1\Hom_k(A\ot{\bullet},M)=\ker\ud^{1,\sigma}/\im\ud^{0,\sigma}$, where 
\begin{align}
    \ker\ud^{1,\sigma}&=\{f:A\rightarrow M\,|\,\forall a,b\!:\!A,f(ab)=af(b)-f(a)b\}=\Der_k(A,M),\notag\\
    \im\ud^{0,\sigma}&=\{f:A\rightarrow M\,|\,\exists m\in\Fix(M),f(a)=am-ma\}.\notag
\end{align} 

Also, $\uHH^0(A,M)_\sigma$ is the zeroth cohomology of the quotient complex $C^n[A\ot{\bullet},M]_{k,\sigma}$. In other words, it is the kernel of the map $\ud^0_\sigma:M_\sigma\rightarrow[A,M]_{k,\sigma}$. To consider this explicitly, we can first consider the map $M\rightarrow[A,M]_{k,\sigma}$ which $\ud^0$ descends to. The kernel of this map has as its underlying set 
\begin{align}
    &\{m\!:\!M|(\ud^0(m)_i)_i\in\im(\sigma-\id)\}\notag\\
    =\:&\{m\!:\!M|\exists h:[A,M]_k\text{ s.t. }\forall i,\forall a,h_{i+1}(a)-h_i(a)=a\sigma^i(m)-\sigma^i(m)a\}\notag\\
    =\:&\{m\!:\!M|\exists\lfloor k\rfloor\text{-hom }h_0:\lfloor A\rfloor\rightarrow\lfloor M\rfloor\text{ s.t. }\sigma_M(h_0(a))=h_0(\sigma_A(a))+\sigma_A(a)m-m\sigma_A(a)\},\notag
\end{align}
the set of $m\in\lfloor M\rfloor$ such that there exists a $\lfloor k\rfloor\text{-hom }h_0$ whose failure to be a difference map is measured by $\ud^0(m)_0\circ\sigma_A$. Then $\ker\ud^0_\sigma$ has as its underlying set the set of quotient classes of $n\in\lfloor M\rfloor$ that differ from such an $m$ by an element of $\im(\sigma_M-\id_M)$. 
\end{remark}

\newpage 

\appendix

\section*{Appendix}\hypertarget{Appendix}{}

Given a topos $\E$, we can form the category Ab$(\E)$ of abelian group objects and their homomorphisms, i.e. the category of models of the theory of abelian groups in $\E$. We write $\Hom_\textbf{Z}(G,H)$ to be the set of group homomorphisms between abelian group objects $G$ and $H$. Given a ring object $A$ in a topos $\E$, we can similarly form the subcategory $A\M$ of all (left) $A$-modules and their homomorphisms. We write $\Hom_A(M,N)$ to be the set of $A$-module homomorphisms between (left) $A$-modules $M$ and $N$. 

In the following, we show that these categories have their own closure, with an internal group hom $[G,H]_\textbf{Z}$ and an internal $A$-hom $[M,N]_A$, and that $A\M$ has a monoidal structure under a tensor product $-\otimes_A-$. We use the internal logic of a topos to set out explicitly what these categories look like, and how to reason in terms of their objects. We recover some classical results in a topos setting, and achieve our main goal of justifying logical arguments on the tensor product in terms of terms of type $m\otimes n\!:\!M\otimes_AN$, where $m\!:\!M,n\!:\!N$. 

In particular, we demonstrate the properties of the tensor product of $A$-modules, with a main goal of justifying arguments on the tensor product in terms of generators. If the category Ab$(\E)$ has a set of cogenerators, then the functor $$\text{Ab}(\E)\to\text{Ab},\quad P\mapsto\Hom_A(M,[N,P]_\textbf{Z})\label{rep}$$ is representable, for any given right $A$-module $M$ and left $A$-module $N$, and so the tensor product exists. For our purposes, this condition is often incidental. The existence of a natural number object is sufficient, but not necessary. For example, the elementary topos \textbf{FinSet}, without N.N.O, has a set of cogenerators for Ab$(\E)$, precisely the set of objects of Ab$(\E)$, and so also satisfies this condition (indeed, one can see that the tensor product of finite modules over a finite ring is itself finite). We then restrict to the case with a N.N.O, and show how resulting free constructions give us a construction of the tensor product. In the case that our topos is also a geometric category, we give an explicit description of the logical structure of the tensor product. 

\section{Internal logic of a topos}

Before we proceed, we need some results from the internal logic of $\E$, as set out in \cite{handbook}. We will make free use of the 96 logical rules, (\textbf{T1}) to (\textbf{T96}), Borceux sets out between Theorem 6.7.1 and Theorem 6.9.6. We also make frequent use of Propositions 6.10.2 and 6.10.9:

\begin{proposition*}[\cite{handbook} Proposition 6.10.2]
    In a topos $\E$, let us consider:
    \begin{itemize}
        \item two objects $A,B$;
        \item two morphisms $f,g:A\rightarrow B$;
        \item three subobjects $A_1\mono A,A_2\mono A,B_1\mono B$;
        \item variables $a,a'$ of type $A$ and $b$ of type $B$.
    \end{itemize}
    The following equivalences hold:
    \begin{enumerate}
        \item $f=g$ iff $\vDash\forall af(a)=g(a)$;
        \item $f$ is a monomorphism iff $\vDash\forall a\forall a'(f(a)=f(a')\implies a=a')$;
        \item $f$ is an epimorphism iff $\vDash\forall b\exists af(a)=b$;
        \item $A_1\subset A_2$ iff $\vDash\forall a(a\in A_1\implies a\in A_2)$;
        \item $f$ factors through $A_1,B_1$ iff $\vDash\forall a(a\in A_1\implies f(a)\in B_1)$.
    \end{enumerate}
\end{proposition*}

\begin{proposition*}[\cite{handbook} Proposition 6.10.9]
    In a topos $\E$, consider a formula $\varphi$ with free variables $a,b$ of types $A,B$. If the relation $\vDash\exists!b\varphi$ holds, there exists a unique morphism $f:A\rightarrow B$ such that $\vDash\varphi(a,f(a))$ holds.
\end{proposition*}

We can show that we satisfy the condition required for use of modus ponens whenever the types of our variables and formulae are group objects. 

\begin{proposition}
    Let $X,Y\in\E$ be inhabited, i.e. with global elements $x:1\to X,y:1\to Y$. Then the projections $\pi_1:X\times Y\to X,\pi_2:X\times Y\to Y$ are epimorphisms. 
\end{proposition}

\begin{proof}
    Consider $X,Y\in\E$ with global elements $x:1\to X,y:1\to Y$. Then $\vDash\forall a\!:\!X(a=\pi_1(a,x))$, and $\vDash\forall b\!:\!Y(b=\pi_2(y,b))$, writing $x,y$ as the variables $x(*),y(*)$ for the unique variable $*\!:\!1$. By \cite{handbook} Proposition 6.10.12, we have that $\pi_1,\pi_2$ are epimorphisms. 
\end{proof}

\begin{corollary}
    Given $M,N\in\text{Grp}(\E)$, the projections $\pi_1:M\times N\to M,\pi_2:M\times N\to N$ are epimorphisms. 
\end{corollary}

\begin{corollary}
    \label{mp}Whenever we consider variables and formulae of group object type (in particular, of $A$-module type), the condition \emph{\textbf{(DR)}} of \emph{\cite{handbook} Theorem 6.7.1} holds, and so we can use modus ponens. 
\end{corollary}

Given $A\in\text{Ring}(\E)$, $M,N\in A\M$ and $f\in\E(M,N)$, we note that the $A$-module homomorphism axioms translate to $$1.\vDash f(m+m')=f(m)+f(m');\quad\quad\quad2.\vDash f(am)=af(m),$$
where $m,m'\!:\!M,a\!:\!A$. Given this, we can also prove the following $A$-module homomorphism analogue of \handbook:

\begin{proposition}
        In a topos $\E$, let $M,N\in A\M$ for some $A\in\text{Ring}(\E)$. Consider a formula $\varphi$ with free variables $m,n$ of types $M,N$. Suppose that, for $m,m'\!:\!M,n,n'\!:\!N,a\!:\!A$, $$1. \vDash(\varphi(m,n)\land\varphi(m',n'))\Rightarrow\varphi(m+m',n+n');\quad2. \vDash\varphi(m,n)\Rightarrow\varphi(am,an).$$
    If the relation $\vDash!n\varphi$ holds, there exists unique $k$-module morphism $f:M\rightarrow N$ such that $\vDash(m,f(n))$.\label{handbook-module}
\end{proposition}

\begin{proof}
    From \handbook, we have unique $f\in\E(M,N)$ such that $\vDash\varphi(m,f(n))$; it remains to show that $f$ is a $A$-module homomorphism. We have $\vDash\varphi(m+m',f(m+m'))$; but as $\vDash\varphi(m,f(m))\land\varphi(m',f(m'))$, then $\vDash\varphi(m+m',f(m)+f(m'))$. By uniqueness, we must have $\vDash f(m+m')=f(m)+f(m')$. Similarly, $\vDash\varphi(am,f(am))$; but as $\vDash\varphi(m,f(m))$, then $\vDash\varphi(am,af(m))$. Again by uniqueness, we obtain $\vDash af(m)=f(am)$. 
\end{proof}

When reasoning with internal logic, for ease of notation we will write $f(y)\!:\!Z$ for the variable $\ev_Y(f,y)$, where $f\!:\![Y,Z],y\!:\!Y$ and $\ev$ denotes the \emph{evaluation} counit $\ev:[Y,Z]\times Y\rightarrow Z$ \cite{kelly}. We can prove an internal analogue to \handbooktwo (a):

\begin{proposition}
\label{internal 6.10.2}Let $X,Y\in\E$, and $f,g:[X,Y]$. Then $$\vDash(f=g)\iff\big(\forall x\!:\!X(f(x)=g(x))\big).$$
\end{proposition}

\begin{proof}
Our formula makes sense in any topos, taking $X,Y$ as arbitrary objects of the topos. So by the localising principle, our formula is valid precisely if it is provable when $f,g$ are global elements $1\stackrel[g]{f}{\rightrightarrows}[X,Y]$, i.e. $f,g\in\Hom(X,Y)$. But by \handbooktwo, $f=g$ if and only if $\forall x:X(f(x)=g(x))$; and so we have our result.
\end{proof}

Finally, we make explicit the logic behind the tensor-internal hom adjuncts of a given map. First, we need to know how to reason about elements of product type: 

\begin{lemma}
    \label{product-type}Given objects $X,Y\in\E$ and variable $a\!:\!X\times Y$, $$\vDash\exists x\!:\!X,y\!:\!Y(a=(x,y)).$$
\end{lemma}

\begin{proof}
    Given a morphism $f:Z\rightarrow X\times Y$, we have that $(\pi_1\circ f,\pi_2\circ f)=f$ \cite{lambek}. Specifically, $(\pi_1,\pi_2)=\id_{X\times Y}$. So $\vDash a=(\pi_1(a),\pi_2(a))$, where $\pi_1(a)\!:\!X,\pi_2(a)\!:\!Y$. 
\end{proof}

\begin{lemma}
\label{adjunction}Let $X,Y,Z\in\E$. Given $f:X\times Y\rightarrow Z$, there exists unique $\hat{f}:X\rightarrow[Y,Z]$ such that, for variables $x\!:\!X,y\!:\!Y$, $$\vDash f(x,y)=\hat{f}(x)(y).$$
\end{lemma}

\begin{proof}
By \cite{adjunction}, the adjunction $\E(X\times Y,Z)\cong\E(X,[Y,Z])$ sends $f:X\times Y\rightarrow Z$ to the composite $$\hat{f}:X\xrightarrow{d}[Y,X\times Y]\xrightarrow{[Y,f]}[Y,Z],$$ where $d$ is the unit of the adjunction. By naturality of $\ev$ with respect to $f$, we have that $\ev_Z\circ([Y,f(-)]\times Y)\circ (d_X(-)\times Y)=f$. In other words, the composite $$\ev(\hat{f},y):X\times Y\xrightarrow{(d,\id)}[Y,X\times Y]\times Y\xrightarrow{([Y,f],\id)}[Y,Z]\times Y\xrightarrow{\ev}Z$$ is equal to $f$. As a result, we have $\vDash f(x,y)=\ev(\hat{f}\times Y)(x,y)$, i.e. $\vDash f(x,y)=\hat{f}(x)(y)$. 
\end{proof}

\begin{corollary}
\label{adjunction-reverse}Let $X,Y,Z\in\E$. Given $g:X\rightarrow[Y,Z]$, there exists unique $\tilde{g}:X\times Y\rightarrow Z$ such that, for variables $x\!:\!X,y\!:\!Y$, $$\vDash g(x)(y)=\tilde{g}(x,y).$$
\end{corollary}

Crucially, we take this unique correspondence $\E(X\times Y,Z)\cong\E(X,[Y,Z]),f\leftrightarrow\hat{f}$, to be our tensor-hom isomorphism. 

\section{Internal homs of abelian groups and modules}

\subsection{Module structure on internal homs}

Let $M,N\in\text{Ab}(\E)$. We begin by describing the group structure on $[M,N]$. We have an addition map $+_{[M,N]}:[M,N]^2\rightarrow[M,N]$ corresponding by adjunction to the composite $$[M,N]^2\times M\xrightarrow{(\id_{[M,N]},\Delta_M)}[M,N]^2\times M^2\cong([M,N]\times M)^2\xrightarrow{\ev_M^2}N^2\xrightarrow{+_N}N.$$

Given $A\in\text{Ring}(\E)$, (right) $A$-module structure on $M$ induces (left) $A$-module structure on $[M,N]$. We have map $\mu_{[M,N]}:A\times [M,N]\rightarrow[M,N]$ corresponding via adjunction to the composite $$A\times [M,N]\times M\cong[M,N]\times M\times A\xrightarrow{(\id_{[M,N]},\mu_M)}[M,N]\times M\xrightarrow{\ev}N.$$ 

Given $B\in\text{Ring}(\E)$, (left) $B$-module structure on $N$ induces (left) $B$-module structure on $[M,N]$. We have map $\mu_{[M,N]}:B\times[M,N]\rightarrow[M,N]$ corresponding via adjunction to the composite $$B\times [M,N]\times M\xrightarrow{(\id_B,\ev)}B\times N\xrightarrow{\mu_N}N.$$ 

From Corollary \ref{adjunction-reverse}, given $f,g\!:\![M,N],m\!:\!M,a\!:\!A,b\!:\!B$, we have $$\vDash(f*a)(m)=f(am),\vDash(b*f)(m)=f(m)b\text{ and }\vDash(f+g)(m)=f(m)+g(m).$$

\subsection{Constructing internal abelian group homs}

Following the ideas of \cite{marty}, we can now define the \emph{internal abelian group hom} of $M,N\in\text{Ab}(\E)$, $$[M,N]_\textbf{Z}:=\text{eq}\big([M,N]\stackrel[+_{N*}]{+_M^*}{\rightrightarrows}[M\times M,N]\big),$$ where $+_M^*=[+_M,N]$ and $+_N^*$ is the adjoint of $$[M,N]\times M^2\xrightarrow{(\Delta_{[M,N]},\id_M^2)}[M,N]^2\times M^2\cong([M,N]\times M)^2\xrightarrow{\ev_M^2}N^2\xrightarrow{+_N}N.$$ 

This comes equipped with evaluation map $[M,N]_\textbf{Z}\times M\rightarrow M$, formed as the composite $\ev_M\circ(i_\textbf{Z},\id_M)$. In general, we will also write this map as $\ev_M$ or $\ev$, and for ease of notation we will write $f(m)\!:\!M$ for
the variable $\ev_M(f,m)$, where $f\!:\![M,N]_\textbf{Z},m\!:\!M$. 

From Corollary \ref{adjunction-reverse}, for $f\!:\![M,N]_\textbf{Z}$ and $m,m'\!:\!M$, we have $\vDash+_M^*(f)(m,m')=f(m+m')$ and $\vDash+_{N*}(f)(m,m')=f(m)+f(m')$. Hence, $$\vDash f(m)+f(m')=f(m+m').$$ 

From \cite{handbook}'s Proposition 6.10.3, we obtain

\begin{proposition}
    We can realise $[M,N]_\textbf{Z}$ as the subobject \label{internal-grp-hom} $$[M,N]_\textbf{Z}=\{f\,|\,\forall m,m'\!:\!M(f(m+m')=f(m)+f(m'))\}\mono[M,N].$$
\end{proposition}

\begin{example}
    \label{classical-grp}The topos $\E=\textbf{Set}$ is enriched in the natural way over hom sets, i.e. we have that $[X,Y]=\Hom(X,Y)$ for any sets $X,Y$. Hence, for two abelian groups $M,N$, we have 
    \begin{align}
        [M,N]_\textbf{Z}&=\{f\,|\,\forall m,m'\in M,f(m+m')=f(m)+f(m')\}\notag\\
        &=\Hom_\mathbb{Z}(M,N).\notag
    \end{align}
\end{example}

\subsection{Closure of Ab(\texorpdfstring{$\E$}{E})}\label{grp-hom-module-structure}

We obtain abelian group structure on $[M,N]_\textbf{Z}$ as the vertical map of 

\[\begin{tikzpicture}[scale=2]
	\node (A1) at (0,1) {$[M,N]_\textbf{Z}$};
	\node (A2) at (0,0) {$[M,N]_\textbf{Z}^2$.};
	\node (B2) at (1,0.5) {$[M,N]^2$};
	\node (C1) at (2,1) {$[M,N]$};
	\node (D1) at (3.5,1) {$[M\times M,N]$};
	\draw[->]
	(B2) edge node [below right] {$+_{[M,N]}$} (C1)
	(A1) edge node [above] {$i_\textbf{Z}$} (C1)
	(A2) edge node [below right] {$i_\textbf{Z}^2$} (B2);
	\draw[dashed,->]
	(A2) edge node [left] {$\exists+_{[M,N]_\textbf{Z}}$} (A1);
	\begin{scope}[transform canvas={yshift=.3em}]
        \draw [->] (C1) edge node [above] {$+^*_M$} (D1);
    \end{scope}
    \begin{scope}[transform canvas={yshift=-.3em}]
        \draw [->] (C1) edge node [below] {$+_{N*}$} (D1);
    \end{scope}
\end{tikzpicture}\]

Given $A\in\text{Ring}(\E)$, if $M$ is a (right) $A$-module, we obtain the left $A$-action map on $[M,N]_\textbf{Z}$ by noting that the $A$-action map on $[M,N]$ factors through $[M,N]_\textbf{Z}\xmono{i_\textbf{Z}}[M,N]$, as for $f\!:\![M,N]_\textbf{Z},a\!:\!A$ and $m,m'\!:\!M$, we have $$\vDash(af)(m+m')=f((m+m')a)=f(ma+m'a)=f(ma)+f(m'a)=(af)(m)+(af)(m').$$ 

Similarly, given $B\in\text{Ring}(\E)$, if $N$ is a (left) $B$-module, we obtain the (left) $B$-action map on $[M,N]_\textbf{Z}$ by noting that the $B$-action map on $[M,N]$ factors through $[M,N]_\textbf{Z}\xmono{i_\textbf{Z}}[M,N]$, as for $f\!:\![M,N]_\textbf{Z},b\!:\!B$ and $m,m'\!:\!M$, we have $$\vDash(b*f)(m+m')=bf(m+m')=b(f(m)+f(m'))=bf(m)+bf(m')=(b*f)(m)+(b*f)(m').$$ 

From this, for $a\!:\!A,b\!:\!B,f,g\!:\![M,N]_\textbf{Z},m\!:\!M$, we obtain $$\vDash(af)(m)=f(ma),\vDash(bf)(m)=f(m)b\text{ and }\vDash(f+g)(m)=f(m)+g(m).$$ 

\begin{proposition}
    \label{ab-functor}The internal abelian group hom forms a functor $$[-,-]_\emph{\textbf{Z}}:\emph{Ab}(\E)^{\emph{op}}\times\emph{Ab}(\E)\to\emph{Ab}(\E).$$ 
\end{proposition}

\begin{proof}
    Let $f:M\to M',g:N\to N'$ be morphisms in $\Ab(\E)$, giving morphism $f\times g:M'\times N\to M\times N'$ in $\Ab(\E)^\text{op}\times\Ab(\E)$. Then, the functor $[-,-]:\E^\text{op}\times\E\to\E$ gives a morphism $$[f,g]:[M',N]\to[M,N'],$$ so that, for variable $h\!:\![M',N]$, we have $$\vDash\forall m\!:\!M([f,g](h)(m)=g(h(f(m)))).$$ Suppose $h\in[M',N]_\textbf{Z}$. Then, 
    \begin{align}
        \vDash\forall m,m'\!:\!M([f,g](h)(m+m')=g(h(f(m+m')))&=g(h(f(m)))+g(h(f(m')))\notag\\&=[f,g](h)(m)+[f,g](h)(m')).\notag
    \end{align} So, $[f,g]$ descends to a morphism $[f,g]_\textbf{Z}:[M',N]_\textbf{Z}\to[M,N']_\textbf{Z}$. The functor conditions follow. 
\end{proof}

\begin{proposition}
    We have natural isomorphism $[1,M]_\textbf{Z}\cong M$. 
\end{proposition}

\begin{proof}
    Define maps $\Phi:[1,M]_\textbf{Z}\to M$ and $\Psi:M\to[1,M]_\textbf{Z}$, given by $$\vDash\forall f\!:\![1,M]_\textbf{Z}(\Phi(f)=f(0_M)),\quad\vDash\forall m\!:\!M(\Psi(m)(e)=m),$$ writing $e$ as the unique variable of type $1$. Then $\vDash\Phi(\Psi(m))=\Psi(m)(0_1)=\Psi(m)(e)=m$, and $\vDash\Psi(\Phi(f))(e)=\Phi(f)=f(0_M)=f(e)$. 
\end{proof}

\begin{proposition}
    We have extranatural transformation $j_M:1\to[M,M]_\textbf{Z}$. 
\end{proposition}

\begin{proof}
    As $\E$ is a closed category, there is an extranatural transformation $j_X:1\to[X,X]$, so that $\vDash\forall x\!:\!X(j_X(e)(x)=x)$, writing $e\!:\!1$ as the unique variable of type $1$. Let $M\in\text{Ab}(\E)$. For variables $m,m'\!:\!M$, we have $$\vDash j_M(1)(m+m')=m+m'=j_M(1)(m)+j_M(1)(m'),$$ and so $j_M(1)\in[M,M]_\textbf{Z}$. So, the extranatural transformation given by the closed structure of $\E$ descends to one on $\Ab(\E)$. 
\end{proof}

\begin{proposition}
    We have transformation $$L^M_{NP}:[N,P]_\textbf{Z}\to[[M,N]_\textbf{Z},[M,P]_\textbf{Z}]_\textbf{Z},$$ natural in $N,P$ and extranatural in $M$. 
\end{proposition}

\begin{proof}
    As $\E$ is a closed category, there is a transformation $$L^X_{YZ}:[Y,Z]\to[[X,Y],[X,Z]],$$ natural in $Y,Z$ and extranatural in $X$. This is such that, for variables $g\!:\![X,Y],x\!:\!X$, $$\vDash\forall f\!:\![Y,Z](L^X_{YZ}(f)(g)(x)=f(g(x))).$$ Let $M,N,P\in\Ab(\E)$. For variables $g,g'\!:\![M,N]_\textbf{Z},m,m'\!:\!M$,
    \begin{align}
        \vDash\forall f\!:\![N,P]_\textbf{Z}(L^M_{NP}(f)(g)(m+m')=f(g(m+m'))&=f(g(m))+f(g(m'))\notag\\&=L^M_{NP}(f)(g)(m)+L^M_{NP}(f)(g)(m')),\notag
    \end{align} and
    \begin{align}
        \vDash\forall f\!:\![N,P]_\textbf{Z}(L^M_{NP}(f)(g+g')(m)=f(g(m)+g'(m))&=f(g(m))+f(g'(m))\notag\\&=L^M_{NP}(f)(g)(m)+L^M_{NP}(f)(g')(m)).\notag
    \end{align}
    So, the transformation given by the closed structure of $\E$ descends to one on $\Ab(\E)$. 
\end{proof}

So, we obtain the following proposition by definition. 

\begin{proposition}
    \label{ab-closed}$\emph{\Ab}(\E)$ is a closed category under the internal abelian group hom. 
\end{proposition}

\subsection{Constructing internal \texorpdfstring{$A$}{A}-homs}

Given $A\in\text{Ring}(\E)$ and (left) $A$-modules $M,N$, we can now define the \emph{internal $A$-hom} $$[M,N]_A:=\text{eq}\big([M,N]_\textbf{Z}\stackrel[\mu_{N*}]{\mu_M^*}{\rightrightarrows}[A\times M,N]_\textbf{Z}\big).$$ Here, $\mu_M^*=[\mu_M,N]$, and $\mu_N^*$ is the adjoint of $$[M,N]\times A\times M\cong  A\times[M,N]\times M\xrightarrow{(\id_A,\ev_M)}A\times N\xrightarrow{\mu_N}N.$$ 

This comes equipped with evaluation map $[M,N]_A\times M\rightarrow M$, formed as the composite $\ev_M\circ(i_A,\id_M)$. In general, we will also write this map as $\ev_M$ or $\ev$, and for ease of notation we will write $f(m)\!:\!M$ for
the variable $\ev_M(f,m)$, where $f\!:\![M,N]_A,m\!:\!M$. 

From Corollary \ref{adjunction-reverse}, for $f\!:\![M,N]_A,m\!:\!M$ and $a\!:\!A$, we have $\vDash\mu_M^*(f)(a,m)=f(am))$ and $\vDash\mu_{N*}(f)(a,m)=af(m)$. Hence, $$\vDash f(am)=af(m).$$ 

From \cite{handbook}'s Proposition 6.10.3, we obtain

\begin{proposition} 
    We can realise $[M,N]_A$ as the subobject
    $$[M,N]_A=\{f\,|\,\forall m\!:\!M,a\!:\!A(f(am)=af(m))\}\mono[M,N]_\textbf{Z}.$$
\end{proposition}

\begin{proposition}\label{internal-k-hom}
    We can realise $[M,N]_A$ as the subobject $$[M,N]_A=\{f\,|\,\forall m,m'\!:\!M,a\!:\!A\big(f(m+m')=f(m)+f(m')\land f(am)=af(m)\big)\}\mono[M,N].$$
\end{proposition}

\begin{example}
    \label{classical-module}Taking $\E=\textbf{Set}$, for a ring $A$ and (left) $A$-modules $M,N$, we have 
    \begin{align}
        [M,N]_k&=\{f\,|\,\forall m,m'\in M,a\in A,f(m+m')=f(m)+f(m')\text{ and } f(am)=af(m)\}\notag\\
        &=\Hom_A(M,N).\notag
    \end{align}
\end{example}

\subsection{Module structure on internal \texorpdfstring{$A$}{A}-homs}



We obtain the group structure on $[M,N]_A$ as the vertical map of

\[\begin{tikzpicture}[scale=2]
	\node (A1) at (0,1) {$[M,N]_A$};
	\node (A2) at (0,0) {$[M,N]_A^2$.};
	\node (B2) at (1,0.5) {$[M,N]_\textbf{Z}^2$};
	\node (C1) at (2,1) {$[M,N]_\textbf{Z}$};
	\node (D1) at (3.5,1) {$[A\times M,N]_\textbf{Z}$};
	\draw[->]
	(B2) edge node [below right] {$+_{[M,N]_\textbf{Z}}$} (C1)
	(A1) edge node [above] {$i_k$} (C1)
	(A2) edge node [below right] {$i_k^2$} (B2);
	\draw[dashed,->]
	(A2) edge node [left] {$\exists+_{[M,N]_A}$} (A1);
	\begin{scope}[transform canvas={yshift=.3em}]
        \draw [->] (C1) edge node [above] {$\mu^*_M$} (D1);
    \end{scope}
    \begin{scope}[transform canvas={yshift=-.3em}]
        \draw [->] (C1) edge node [below] {$\mu_{N*}$} (D1);
    \end{scope}
\end{tikzpicture}\]

From this, for $a\!:\!A,f,g\!:\![M,N]_A,m\!:\!M$, we obtain $$\vDash(f+g)(m)=f(m)+g(m).$$ 

Given $B\in\text{Ring}(\E)$, if $M$ is a (right) $B$-module, we obtain the (left) $B$-action map on $[M,N]_A$ by noting that the $B$-action map on $[M,N]_\textbf{Z}$ factors through $[M,N]_A\xmono{i_A}[M,N]_\textbf{Z}$, as for $f\!:\![M,N]_A,a\!:\!A,b\!:\!B$ and $m\!:\!M$, we have $$\vDash(bf)(am)=f(amb)=af(mb)=a(bf)(m).$$ 

\begin{proposition}
    \label{mod-functor}The internal $A$-hom forms a functor $$[-,-]_A:A\emph{\M}^{\emph{op}}\times A\emph{\M}\to\emph{\Ab}(\E).$$ 
\end{proposition}

\begin{proof}
    Let $f:M\to M',g:N\to N'$ be morphisms in $A\M$, giving morphism $f\times g:M'\times N\to M\times N'$ in $A\M^\text{op}\times A\M$. Then, by Proposition \ref{ab-functor} the functor $[-,-]_\textbf{Z}:\Ab(\E)^\text{op}\times\Ab(\E)\to\Ab(\E)$ gives a morphism $$[f,g]_\textbf{Z}:[M',N]_\textbf{Z}\to[M,N']_\textbf{Z},$$ so that, for variable $h\!:\![M',N]_\textbf{Z}$, we have $$\vDash\forall m\!:\!M([f,g]_\textbf{Z}(h)(m)=g(h(f(m)))).$$ Suppose $h\in[M',N]_A$. Then, 
    \begin{align}
        \vDash\forall m\!:\!M,a\!:\!A([f,g]_\textbf{Z}(h)(am)=g(h(f(am)))&=g(ah(f(m)))\notag\\&=[f,g]_\textbf{Z}(ah)(m).\notag
    \end{align} So, $[f,g]_\textbf{Z}$ descends to a morphism $[f,g]_A:[M',N]_A\to[M,N']_A$. The functor conditions follow. 
\end{proof}

\begin{proposition}
    \label{[A,M]}We have natural isomorphism $[A,M]_A\cong M$, where $[A,M]_A$ obtains its action from the (right) action on $A$. 
\end{proposition}

\begin{proof}
    Define maps $\Phi:[A,M]_A\to M$ and $\Psi:M\to[A,M]_A$, given by $$\vDash\forall f\!:\![A,M]_A(\Phi(f)=f(1_A)),\quad\vDash\forall m\!:\!M,a\!:\!A(\Psi(m)(a)=am).$$ $\Phi$ is an $A$-module homomorphism, as, for $f,f'\!:\![A,M]_A,a\!:\!A$, $\Phi(af)=(af)(1_A)=af(1_A)=a\Phi(f)$, and $\Phi(f+f')(1_A)=f(1_A)+f'(1_A)$. Then, for $m\!:\!M,f\!:\![A,M]_A$, $\vDash\Phi(\Psi(m))=\Psi(m)(1_A)=1_Am=m$, and $\vDash\forall a\!:\!A(\Psi(\Phi(f))(a)=a\Phi(f)=af(1_A)=f(a))$. 
\end{proof}

\subsection{Closure of \texorpdfstring{$k$}{k}-Mod}

Let $k\in\text{Ring}(\E)$ be a commutative ring. Then every left $k$-module is also a right $k$-module. In particular, for every $M,N\in k\M$, we obtain a $k$-module structure on $[M,N]_k$, such that, for variables $\lambda\!:\!k,f\!:\![M,N]_k,m\!:\!M$, we have $\vDash\lambda f(m)=f(\lambda m)=\lambda f(m)$. In this case, we obtain the following results: 

\begin{proposition}
    \label{com-mod-functor}The internal $k$-hom forms a functor $$[-,-]_k:k\emph{\M}^{\emph{op}}\times k\emph{\M}\to k\emph{\M}.$$ 
\end{proposition}

\begin{proposition}
    We have extranatural transformation $j_M:k\to[M,M]_k$. 
\end{proposition}

\begin{proof}
    

    Define $j_M:k\to[M,M]$ to correspond via adjunction to $\mu_M:k\times M\to M$. By \ref{adjunction}, we have $\vDash\forall m\!:\!M,\lambda\!:\!k(j_M(\lambda)(m)=\lambda m)$. For variables $\lambda,\lambda'\!:\!k,m,m'\!:\!M$, we have $\vDash j_M(\lambda+\lambda')=(\lambda+\lambda')m=\lambda m+\lambda'm=j_M(\lambda)+j_M(\lambda')$ and $\vDash j_M(\lambda\lambda')(m)=\lambda\lambda'm=\lambda j_M(\lambda')(m)$, so $j_M$ is a $k$-module homomorphism. Similarly, $\vDash j_M(\lambda)(m+m')=\lambda(m+m')=\lambda m+\lambda m'=j_M(\lambda)(m)+j_M(\lambda)(m')$ and $\vDash j_M(\lambda)(\lambda'm)=\lambda\lambda'm=\lambda'(\lambda m)=\lambda j_M(\lambda)(m)$. So, $j_M$ factors through $[M,M]_k$. 
\end{proof}

\begin{proposition}
    We have transformation $$L^M_{NP}:[N,P]_k\to[[M,N]_k,[M,P]_k]_k,$$ natural in $N,P$ and extranatural in $M$. 
\end{proposition}

\begin{proof}
    By Proposition \ref{ab-closed}, $\Ab(\E)$ is a closed category, so there is a transformation $$L^M_{NP}:[N,P]_\textbf{Z}\to[[M,N]_\textbf{Z},[M,P]_\textbf{Z}]_\textbf{Z},$$ natural in $N,P$ and extranatural in $M$. This is such that, for variables $g\!:\![M,N],m\!:\!M$, $$\vDash\forall f\!:\![N,P]_\textbf{Z}(L^M_{NP}(f)(g)(m)=f(g(m))).$$ Let $M,N,P\in A\M$. For variables $g\!:\![M,N]_k,m\!:\!M,\lambda\!:\!k$,
    \begin{align}
        \vDash\forall f\!:\![N,P]_k(L^M_{NP}(f)(g)(\lambda m)=f(g(\lambda m))&=f(\lambda g(m))=L^M_{NP}(f)(\lambda g)(m)\notag\\&=(\lambda f)(g(m))=L^M_{NP}(\lambda f)(g)(m)\notag\\&=\lambda f(g(m))=\lambda L^M_{NP}(f)(g)(m)).\notag
    \end{align}
    So, the transformation given by the closed structure of $\Ab(\E)$ descends to one on $k\M$. 
\end{proof}

So, we obtain the following proposition by definition. 

\begin{proposition}
    \label{mod-closed}Given commutative ringed topos $(\E,k)$, $k\emph{\M}$ is a closed category under the internal $k$-hom. 
\end{proposition}

\section{Tensor product of modules}

\subsection{Tensor product functor}

Let $\E$ be a topos such that, given $A\in\text{Ring}(\E)$ with right $A$-module $M$ and left $A$-module $N$, the functor
$$\text{Ab}(\E)\rightarrow\text{Ab},\quad P\mapsto\Hom_A(M,[N,P]_\textbf{Z})$$
is representable (for example, any $\E$ such that $\text{Ab}(\E)$ has a cogenerating set, by the Special Adjoint Functor Theorem). Write the representing object as $M\otimes_AN\in\text{Ab}(\E)$, i.e. $$\Hom_\textbf{Z}(M\otimes_AN,P)\cong\Hom_k(M,[N,P]_\textbf{Z}).$$

This isomorphism is immediately natural in $P$, but we also have naturalilty in $M$. Given a map $f:M\to M'$ of right $A$-modules, we have a map $\Hom_\textbf{Z}(M\otimes_AN,P)\to\Hom_\textbf{Z}(M'\otimes_AN,P)$ induced by the diagram
\[\begin{tikzpicture}[scale=2]
    \node (A1) at (0,0.8) {$\Hom_\textbf{Z}(M\otimes_AN,P)$};
    \node (A2) at (0,0) {$\Hom_\textbf{Z}(M'\otimes_AN,P)$};
    \node (B1) at (1.8,0.8) {$\Hom_A(M,[N,P]_\textbf{Z})$};
    \node (B2) at (1.8,0) {$\Hom_A(M',[N,P]_\textbf{Z})$};
    \draw[->]
    (A1) edge node [left] {} (A2)
    (B1) edge node [right] {$f^*$} (B2);
    \draw[-]
    (A1) edge node [above] {$\sim$} (B1)
    (A2) edge node [above] {$\sim$} (B2);
\end{tikzpicture}\]
So, for all left $A$-modules $N$, this yields a functor $-\otimes_AN:M\mapsto M\otimes_AN$, with an adjunction 
\[\begin{tikzpicture}[scale=1]
	\node (A1) at (0,0) {Mod-$A$};
	\node (B1) at (1.4,0) {\rotatebox[origin=c]{90}{$\vdash$}};
	\node (C1) at (2.8,0) {Ab$(\E)$};
	\draw[->,bend left=20]
	(A1) edge node [above] {$-\otimes_AN$} (C1)
	(C1) edge node [below] {$[N,-]_\textbf{Z}$} (A1);
\end{tikzpicture}\]
We obtain map $\otimes:M\times N\to M\otimes_kN$ from the unit $\eta$ of the above adjunction via \begin{align}
    \eta_M\in\Hom_k(M,[N,M\otimes_kN]_\textbf{Z})\mono\Hom_k(M,[N,M\otimes_kN])\mono\,&\E(M,[N,M\otimes_kN])\notag\\\cong\,&\E(M\times N,M\otimes_kN),\notag
\end{align} where the first map is $\Hom_\textbf{Z}(M,i)$, an injection by the left exactness of $\Hom_\textbf{Z}(M,-)$. Hence, $\vDash\forall m\!:\!M,n\!:\!N(m\otimes n=\eta(m)(n))$. 

\begin{example}
    \label{classical-tensor}From Example $\ref{classical-module}$, we can see that setting $\E=\textbf{Set}$ we recover the classical tensor product of modules. 
\end{example}

\subsection{Universal property of the tensor product}

Let $A\in\E$ be a ring object; $M$ a right $A$-bimodule, $N$ a left $A$-bimodule, and $P\in\text{Ab}(\E)$. We show the universal property of the tensor product, via analogous reasoning to \cite{conrad}. 

We say $f:M\times N\rightarrow P$ is a \emph{balanced product} (in $M$ and $N$) if the following diagrams commute: \[\begin{tikzpicture}[scale=2]
    \node (A1) at (0,0.8) {$M^2\times N$}; 
    \node (A2) at (0,0) {$(M\times N)^2$};
    \node (B2) at (1,0) {$P^2$};
    \node (C1) at (1.8,0.8) {$M\times N$};
    \node (C2) at (1.8,0) {$P$};
    \draw[->]
        (A1) edge node [above] {$(+_M,\id_N)$} (C1)
        (C1) edge node [right] {$f$} (C2)
        (A1) edge node [left] {$(\id_M^2,\Delta_N)$} (A2)
        (A2) edge node [above] {$f^2$} (B2)
        (B2) edge node [above] {$+_P$} (C2);
\end{tikzpicture}
\begin{tikzpicture}[scale=2]
    \node (A1) at (0,0.8) {$M\times N^2$}; 
    \node (A2) at (0,0) {$(M\times N)^2$};
    \node (B2) at (1,0) {$P^2$};
    \node (C1) at (1.8,0.8) {$M\times N$};
    \node (C2) at (1.8,0) {$P$};
    \draw[->]
        (A1) edge node [above] {$(\id_M,+_N)$} (C1)
        (C1) edge node [right] {$f$} (C2)
        (A1) edge node [left] {$(\Delta_M,\id_N^2)$} (A2)
        (A2) edge node [above] {$f^2$} (B2)
        (B2) edge node [above] {$+_P$} (C2);
\end{tikzpicture}\] \[\begin{tikzpicture}[scale=2]
    \node (A1) at (0,0.8) {$M\times A\times N$}; 
    \node (A2) at (0,0) {$M\times N$};
    \node (B1) at (1.5,0.8) {$M\times N$};
    \node (B2) at (1.5,0) {$P$};
    \draw[->]
        (A1) edge node [above] {$(\id_M,\mu_N)$} (B1)
        (B1) edge node [right] {$f$} (B2)
        (A1) edge node [left] {$(\mu_M,\id_N)$} (A2)
        (A2) edge node [above] {$f$} (B2);
\end{tikzpicture}\]

In other words, if we have 

\begin{enumerate}
  \itemB $\vDash\forall m,m'\!:\!M,n\!:\!N\big(f(m+m',n)=f(m,n)+f(m',n)\big)$;
  \itemB $\vDash\forall m\!:\!M,n,n'\!:\!N\big(f(m,n+n')=f(m,n)+f(m,n')\big)$;
  \itemB $\vDash\forall m\!:\!M,n\!:\!N,a\!:\!A\big(f(ma,n)=f(m,an)\big)$.
\end{enumerate}
Write $\text{Bil}_A(M,N;P)$ as the set of balanced products $M\times N\to P$.

\begin{lemma}
    The set $\emph{Bil}_A(M,N;P)$ is an abelian group, inducing functor $$\emph{Bil}_A(M,N;-):\emph{Ab}(\E)\rightarrow\emph{Ab}.$$
\end{lemma}

\begin{proof}
    We use variables $m,m'\!:\!M,n,n'\!:\!N$. Given $f,f':M\times N\rightarrow P$, we have $$\vDash\exists!p\!:\!P(p=f(m,n)+f'(m,n)).$$ So, from \cite{handbook} Proposition 6.10.9, we have morphism $(f+f'):M\times N\rightarrow P$ such that $\vDash(f+f')(m,n)=f(m,n)+f'(m,n)$. Suppose $f,f'$ are balanced products. We can show that $f+f'$ is also a balanced product - for example, $\vDash(f+f')(m+m',n)=f(m+m',n)+f'(m+m',n)=f(m,n)+f(m',n)+f'(m,n)+f'(m',n)=(f+f')(m,n)+(f+f')(m',n)$. The group axioms are simple to verify. 
\end{proof}

\begin{lemma}
    Let $M$ be a right $A$-module, $N$ a left $A$-module. The map $\otimes:M\times N\to M\otimes_AN$ is a balanced product. 
\end{lemma}

\begin{proof}
    Take $i:[N,M\otimes_AN]_\textbf{Z}\mono[N,M\otimes_AN]$. We have that $\eta$, and therefore $i\circ\eta$, is a group homomorphism. So, by the explicit adjunction, for variables $m,m'\!:\!M,n\!:\!N$, $$\vDash(m+m')\otimes n=i(\eta(m+m'))(n)=i(\eta(m))(n)+i(\eta(m'))(n)=m\otimes n+m'\otimes n.$$ Similarly, for all $m\!:\!M$ we have $\eta(m)\in[N,M\otimes_AN]_\textbf{Z}$, and therefore $i(\eta(m))\in[N,M\otimes_AN]_\textbf{Z}$. So, for variables $m\!:\!M,n,n'\!:\!N$, $$\vDash m\otimes(n+n')=i(\eta(m))(n+n')=i(\eta(m))(n)+i(\eta(m))(n')=m\otimes n+m\otimes n'.$$ Finally, note that for right $A$-module $N$ and abelian group object $P$, the $A$-module structure on $\alpha\!:\![N,P]_\textbf{Z}$ comes from $\vDash\forall a\!:\!A\forall n\!:\!N((a\alpha)(n)=\alpha(an))$. So, again as $\eta$ is a $k$-module homomorphism, for variables $a\!:\!A,m\!:\!M,n\!:\!N$, $$\vDash ma\otimes n=i(\eta(ma))(n)=a(\eta(m))(n)=\eta(m)(an)=m\otimes an.$$
\end{proof}

\begin{lemma}
    \label{balanced}Given a balanced product $f:M\times N\rightarrow P$ and a group homomorphism $g:P\rightarrow Q$, $g\circ f$ is a balanced product.
\end{lemma}

The proof of this argues via variables exactly how the classical prove argues via elements. 

\begin{lemma}
    We have natural identification $\emph{Bil}_A(M,N;P)\cong\Hom_A(M,[N,P]_\textbf{Z})$. 
\end{lemma}

\begin{proof}
    By our work in internal logic, we know that given a map $f:M\times N\rightarrow P$, there is a unique map $\hat{f}:M\rightarrow[N,P]$ such that $\vDash\forall m\!:\!M,n\!:\!N,f(m,n)=\hat{f}(m)(n)$. For variables $m\!:\!M,n,n'\!:\!N$, $\vDash f(m,n+n')=f(m,n)+f(m,n')\text{ if and only if }\vDash\hat{f}(m)(n+n')=\hat{f}(m)(n)+\hat{f}(m)(n').$ By Proposition \ref{internal-grp-hom}, $[N,P]_\textbf{Z}=\{f|f(n+n')=f(n)+f(n')\}$. So, B2 holds if and only if $\hat{f}$ factors through $[N,P]_\textbf{Z}$. 

    Given such an $f$ where B2 holds, let $m,m'\!:\!M,n\!:\!N,a\!:\!A$. By Lemma \ref{internal-grp-hom} and Subsection \ref{grp-hom-module-structure}, $\hat{f}$ is $k$-linear if and only if $\vDash a\hat{f}(m)=\hat{f}(ma)$ and $\vDash\hat{f}(m+m')=\hat{f}(m)+\hat{f}(m')$. From Proposition \ref{internal 6.10.2}, substituting the types $N,P$ for $X,Y$ and the variables $a\hat{f}(m)$ and $\hat{f}(ma)$ for $f,g$, $$\vDash a\hat{f}(m)=\hat{f}(ma)\Leftrightarrow\forall n\!:\!N(a\hat{f}(m)(n)=\hat{f}(ma)(n)).$$ Also, $\vDash a\hat{f}(m)(n)=\hat{f}(ma)(n)\Leftrightarrow\hat{f}(m)(an)=\hat{f}(ma)(n)\Leftrightarrow f(m,an)=f(ma,n)$. By \cite{handbook} Proposition 6.8.3 (T54),
    \begin{align}
        \label{B3}\vDash\forall a\!:\!A,m\!:\!M(a\hat{f}(m)=\hat{f}(ma))\Leftrightarrow\forall a\!:\!A,m\!:\!M\big(\forall n\!:\!N(f(m,an)=f(ma,n))\big).
    \end{align} Similarly, From Proposition \ref{internal 6.10.2}, substituting the types $N,P$ for $X,Y$ and the variables $\hat{f}(m+m')$ and $\hat{f}(m)+\hat{f}(m')$ for $f,g$, $$\vDash \hat{f}(m+m')=\hat{f}(m)+\hat{f}(m')\Leftrightarrow\forall n\!:\!N(\hat{f}(m+m')(n)=\hat{f}(m)(n)+\hat{f}(m')(n)).$$ Also, $\vDash \hat{f}(m+m')(n)=\hat{f}(m)(n)+\hat{f}(m')(n)\Leftrightarrow f(m+m',n)=f(m,n)+f(m',n)$. By \cite{handbook} Proposition 6.8.3 (T54),
    \begin{align}
        \label{B1}\vDash&\forall m,m'\!:\!M(\hat{f}(m+m')=\hat{f}(m)+\hat{f}(m'))\notag\\&\Leftrightarrow\forall m,m'\!:\!M\big(\forall n\!:\!N(f(m+m',n)=f(m,n)+f(m',n))\big).
    \end{align} By Corollary \ref{mp}, we can apply modus ponens to statements (\ref{B3}) and (\ref{B1}) to obtain that $\hat{f}$ is $k$-linear if and only if $\vDash\forall m\!:\!M,n\!:\!N,a\!:\!A(f(m,an)=f(ma,n))$ and $\vDash\forall m,m'\!:\!M,n\!:\!N(f(m+m',n)=f(m,n)+f(m',n))$, i.e. if and only if $f$ is a balanced product. 
\end{proof}

\begin{corollary}
    We have isomorphism $\emph{Bil}_A(M,N;P)\cong\Hom_\textbf{Z}(M\otimes_AN,P)$. 
\end{corollary}

We recover the map $\otimes:M\times N\to M\otimes_AN$ as a bilinear map corresponding to $\id_{M\otimes_AN}$ via the isomorphism $$\text{Bil}_A(M,N;M\otimes_AN)\cong\Hom_\textbf{Z}(M\otimes_AN,M\otimes_AN),$$ or equivalently as corresponding to the counit $\eta$ via the isomorphism $$\text{Bil}_A(M,N;M\otimes_AN)\cong\Hom_A(M,[N,M\otimes_AN]_\textbf{Z}).$$

\begin{proposition}[Universal property of the tensor product]
    Let $M$ be a right $A$-module $M$, $N$ a left $A$-module. For every abelian group object $P$ and every balanced product $$f:M\times N\rightarrow P,$$ there is a unique group homomorphism $$\tilde{f}:M\otimes_AN\rightarrow P$$ such that $\tilde{f}\circ\otimes=f$. 
\end{proposition}

\begin{proof}
    This simply makes explicit the identification $\text{Bil}_A(M,N;P)\cong\Hom_\textbf{Z}(M\otimes_AN,P)$. For every balanced product $f\in \text{Bil}_A(M,N;P)$, there is a unique group homomorphism $\tilde{f}\in\Hom_\textbf{Z}(M\otimes_AN,P)$; and for every group homomorphism $\tilde{g}\in\Hom_\textbf{Z}(M\otimes_AN,P)$, we identify it with $g\circ\otimes$; as $\otimes$ is balanced, its composition with homomorphism $g$ is balanced.
\end{proof}

\subsection{Internal logic of the tensor product}

\begin{proposition}
    \label{handbooktwo-tensor}Let $M$ a right $A$-module, $N$ a left $A$-module, $P\in\emph{Ab}(\E)$. Given group homomorphisms $f,g:M\otimes_kN\rightarrow P$, we have $$f=g\iff\vDash\forall m\!:\!M,n\!:\!N,\big(f(m\otimes n)=g(m\otimes n)\big).$$
\end{proposition}

\begin{proof}
    By Lemma \ref{product-type}, $\vDash\forall m\!:\!M,n\!:\!N,\big(f(m\otimes n)=g(m\otimes n)\big)$ implies that $f\circ\otimes=g\circ\otimes:M\times N\rightarrow P$. We can recover $f,g$ from $f\circ\otimes,g\circ\otimes$ via the universal property of the tensor product; by the uniqueness, $f\circ\otimes=g\circ\otimes$ implies $f=g$. 
\end{proof}

\begin{proposition}
    \label{tensor-identities}The following hold, where $m,m'\!:\!M,n,n'\!:\!N$ and $a\!:\!A$:
    \begin{enumerate}
        \item $\vDash ma\otimes n=m\otimes an$;
        \item $\vDash(m+m')\otimes n=m\otimes n+m'\otimes n$; 
        \item $\vDash m\otimes(n+n')=m\otimes n+m\otimes n'$. 
    \end{enumerate}
\end{proposition}

\begin{proof}
    This is immediate from the bilinearity of $\otimes:M\times N\to M\otimes_kN$.
\end{proof}

\begin{proposition}
    \label{tensor-module}If $M$ is a right $A$-module, $N$ a left $A$-module, then $M\otimes_AN$ is a (right) $A$-module.
\end{proposition}

\begin{proof}
    We know from the bilinearity of $\otimes$ that it equalises $\mu_M,\mu_N:M\times A\times N\to M\times N$, i.e. $\otimes\circ r_M=\otimes\circ l_N:M\times A\times N\to M\otimes_AN$. Then we obtain corresponding $\tilde{\mu}:M\times N\to[A,M\otimes_AN]$ via the adjunction, such that $\vDash\tilde{\mu}(m,n)(a)=\mu(m,n,a)=ma\otimes n=m\otimes an$, for $m\!:\!M,n\!:\!N,a\!:\!A$. We show that $\tilde{\mu}$ is a balanced product. For variables $m,m'\!:\!M,n,n'\!:\!N,a'\!:\!A$, 
    \begin{enumerate}
        \item $\vDash\forall a\!:\!A\big(\tilde{\mu}(m+m',n)(a)\!=\!(m+m')\otimes an\!=\!ma\otimes n+m'a\otimes n\!=\!\tilde{\mu}(m,n)(a)+\tilde{\mu}(m',n)(a)\big)$.
        \item $\vDash\forall a\!:\!A\big(\tilde{\mu}(m,n+n')(a)=ma\otimes(n+n')=ma\otimes n+ma\otimes n'=\tilde{\mu}(m,n)(a)+\tilde{\mu}(m,n')(a)\big)$.
        \item $\vDash\forall a\!:\!A\big(\tilde{\mu}(ma',n)(a)=mab\otimes n=m\otimes a'an=\tilde{\mu}(m,a'n)(a)\big)$.
    \end{enumerate} 
    The first of these is equivalent to $\vDash\forall m,m'\!:\!M,n\!:\!N\big(\forall a\!:\!A(\tilde{\mu}(m+m',n)(a)=\tilde{\mu}(m,n)(a)+\tilde{\mu}(m',n)(a))\big)$. From Proposition \ref{internal 6.10.2}, substituting the types $M,N$ for $X,Y$ and the variables $\tilde{\mu}(m+m',n)$ and $\tilde{\mu}(m,n)+\tilde{\mu}(m',n)$ for $f,g$, $$\vDash\forall a\!:\!A(\tilde{\mu}(m+m',n)(a)=\tilde{\mu}(m,n)(a)+\tilde{\mu}(m',n)(a))\Leftrightarrow(\tilde{\mu}(m+m',n)=\tilde{\mu}(m,n)+\tilde{\mu}(m',n)).$$ By \cite{handbook} Proposition 6.8.3 (T54), 
    \begin{align}
        \vDash&\forall m,m'\!:\!M,n\!:\!N,\big(\forall a\!:\!A(\tilde{\mu}(m+m',n)(a)=\tilde{\mu}(m,n)(a)+\tilde{\mu}(m',n)(a))\big)\notag\\&\Leftrightarrow\forall m,m'\!:\!M,n\!:\!N,(\tilde{\mu}(m+m',n)=\tilde{\mu}(m,n)+\tilde{\mu}(m',n)).\notag
    \end{align} By Corollary \ref{mp}, we can apply modus ponens to obtain $$\vDash\forall m,m'\!:\!M,n\!:\!N,(\tilde{\mu}(m+m',n)=\tilde{\mu}(m,n)+\tilde{\mu}(m',n)),$$ the first balanced product axiom. Through similar logic, we obtain the second and third, to show that $\tilde{\mu}$ is a balanced product. So, by the universal property of the tensor product, $\hat{\mu}$ induces a group homomorphism $M\otimes_AN\to[A,M\otimes_AN]$. We take the right $A$-module structure to be the corresponding morphism $M\otimes_AN\times A\to M\otimes_AN$ such that $(m\otimes n)a=m\otimes an=ma\otimes n$. 
\end{proof}

\begin{proposition}
    If $M$ is a right $A$-module and $N,P$ are left $A$-modules, then there is a canonical isomorphism
    $$\Hom_A(M\otimes_AN,P)\cong\Hom_A(M,[N,P]_A).$$
\end{proposition}

\begin{proof}
    Consider an $A$-module homomorphism $f:M\otimes_AN\to P$. This is a group homomorphism, so corresponds to a bilinear map $f\circ\otimes:M\times N\to P$. This then corresponds to a map $\hat{f}:M\to[N,G]_\textbf{Z}$ such that, for $m\!:\!M,n\!:\!N$, $\vDash\hat{f}(m)(n)=f(m\otimes n)$. But $f$ is an $A$-module homomorphism, so for $a\!:\!A$, $$\vDash\hat{f}(m)(an)=f(m\otimes an)=f(ma\otimes n)=\hat{f}(ma)(n)=a\hat{f}(m)(n).$$ Hence, for all $m\!:\!M$, we have $\hat{f}(m)\!:\![N,P]_A$, and so $\hat{f}$ factors through $M\to[N,P]_A$. 
\end{proof}

\begin{corollary}
    \label{induced A-hom} Given right $A$-module $M$ and left $A$-modules $N,P$, let bilinear map $f:M\times N\to P$ induce $\tilde{f}:M\otimes_AN\to P$. Writing variables $m\!:\!M,n\!:\!N$,
\begin{align}
    \tilde{f}\text{ is an }A\text{-module homomorphism}\quad\iff\quad\vDash\forall a\!:\!A(f(m,an)=af(m,n)),\notag
\end{align}  
\end{corollary}



\begin{corollary}
    \label{monoidal-closed}Given commutative $k\in\text{Ring}(\E)$, the category $(k\M,\otimes_k)$ is monoidal closed.
\end{corollary}

\begin{proposition}
    If $N$ is an $A$-$B$-bimodule, then $M\otimes_AN$ is a right $B$-module. If $M$ is a $B$-$A$-bimodule, then $M\otimes_AN$ is a right $B$-module.
\end{proposition}

\begin{proof}
    We prove the first statement. Let $\mu:M\times N\times B\to M\otimes_AN$ be the composite $M\times N\times B\xrightarrow{(\id_M,r_N)}M\times N\xrightarrow{\otimes}M\otimes_AN$. Let $\hat{\mu}$ be the correspondent map $M\times N\to[B,M\otimes_AN]$, so that for all $m\!:\!M,n\!:\!N,b\!:\!B$ we have $\vDash\hat{\mu}(m,n)(b)=\mu(m,n,b)=m\otimes nb$. We show that $\hat{\mu}$ is a balanced product. For variables $m,m'\!:\!M,n,n'\!:\!N,a\!:\!A$, we have 
    \begin{enumerate}
        \item $\vDash\forall b\!:\!B\big(\hat{\mu}(m+m',n)(b)=(m+m')\otimes nb=m\otimes nb+m'\otimes nb=\hat{\mu}(m,n)(b)+\hat{\mu}(m',n)(b)\big)$.
        \item $\vDash\forall b\!:\!B\big(\hat{\mu}(m,n+n')(b)=m\otimes(n+n')b=m\otimes(nb+n'b)=m\otimes nb+m\otimes n'b=\hat{\mu}(m,n')(b)+\hat{\mu}(m,n')(b)\big)$.
        \item $\vDash\forall b\!:\!B\big(\hat{\mu}(ma,n)(b)=(ma)\otimes(nb)=m\otimes anb=\hat{\mu}(m,an)(b)\big)$.
\end{enumerate} By the logic of the proof of Proposition \ref{tensor-module}, we obtain the balanced product axioms  
    \begin{enumerate}
        \item $\vDash\hat{\mu}(m+m',n)=\hat{\mu}(m,n)+\hat{\mu}(m',n)$.
        \item $\vDash\hat{\mu}(m,n+n')=\hat{\mu}(m,n)+\hat{\mu}(m,n')$.
        \item $\vDash\hat{\mu}(ma,n)=\hat{\mu}(m,an)$.
    \end{enumerate} So, by the universal property of the tensor product, $\hat{\mu}$  induces a group homomorphism $M\otimes_AN\to[B,M\otimes_AN]$. We take the right $B$-module structure to be the corresponding morphism $M\otimes_AN\times B\to M\otimes_AN$ such that $(m\otimes n)b=m\otimes nb$. The proof of the second statement is the same, using the symmetry of the cartesian product in $\E$. 
\end{proof}

\begin{proposition}
    In a topos $\E$, let $M$ be a right $A$-module, $N$ a left $A$-module for some $A\in\text{Ring}(\E)$, and let $P\in\text{Ab}(\E)$. Consider a formula $\varphi$ with free variables $m,n,p$ of type $M,N,P$. We can think of $\varphi'((m,n)\!:\!M\times N,p\!:\!P)=\varphi(m,n,p)$. Suppose that, for variables $m,m'\!:\!M,n,n'\!:\!N,p,p'\!:\!P,a\!:\!A$ 
    \begin{enumerate}
        \item $\vDash(\varphi(m,n,p)\land\varphi(m',n,p'))\implies\varphi(m+m',n,p+p')$; 
        \item $\vDash(\varphi(m,n,p)\land\varphi(m,n',p'))\implies\varphi(m,n+n',p+p')$; 
        \item $\vDash\varphi(ma,n,p)\implies\varphi(m,an,p)$; 
    \end{enumerate} If the relation $\vDash\exists!b\varphi$ holds, there exists unique morphism $\hat{f}:M\otimes_AN\rightarrow P$ such that $\vDash\varphi(m,n,\hat{f}(m\otimes n))$.\label{handbook-tensor}
\end{proposition}

\begin{proof}
    Suppose we have as above. Then, from Lemma \ref{product-type}, as $\vDash\exists!p\varphi$ holds, $\vDash\exists!p\varphi'$ holds. So, from \handbook, there exists unique morphism $f\in\E(M\times N,P)$ such that $\vDash\varphi'((m,n),f(m,n))$. We wish to show that $f$ is a balanced product. We have $\vDash\varphi'((m+m',n),f(m+m',n))$; but as $\vDash\varphi'((m,n),f(m,n))\land\varphi'((m',n,f(m',n))$, then by Corollary \ref{mp} we can apply modus ponens to obtain $\vDash\varphi'((m+m',n),f(m,n)+f(m',n))$. By uniqueness, we must have $\vDash f(m+m',n)=f(m,n)+f(m',n)$. The argument follows for the second identity. Similarly, $\vDash\varphi'((ma,n),f(am,n))$ and $\vDash\varphi((m,an),f(m,an))$, so $\vDash\varphi((m,an),f(ma,n))$. By uniqueness, we must have $\vDash f(ma,n)=f(m,an)$, i.e. $f$ is a balanced product. So, by the universal property of the tensor, there exists unique $k$-module homomorphism $\hat{f}:M\otimes_kN$ such that $\hat{f}\circ\otimes=f$, i.e., $\vDash\varphi(m,n,\hat{f}(m\otimes n))$. 
\end{proof}

\begin{proposition}
    Let $M,P$ be right $A$-modules, $N,Q$ left $A$-modules for some $A\in\text{Ring}(\E)$. Given $A$-module homomorphisms $f:M\rightarrow P,g:N\rightarrow Q$, we have unique $A$-module homomorphism $(f\otimes g):M\otimes_AN\rightarrow P\otimes_AQ$ such that the following commutes:
    \[\begin{tikzpicture}[scale=2]
        \node (A1) at (0,0.8) {$M\times N$};
        \node (A2) at (0,0) {$M\otimes_AN$};
        \node (B1) at (1.2,0.8) {$P\times Q$};
        \node (B2) at (1.2,0) {$P\otimes_AQ$};
        \draw[->]
            (A1) edge node [left] {$\otimes$} (A2)
            (A1) edge node [above] {$(f,g)$} (B1)
            (A2) edge node [above] {$f\otimes g$} (B2)
            (B1) edge node [right] {$\otimes$} (B2);
    \end{tikzpicture}\]
\end{proposition}

\begin{proof}
    We have that the top path $\otimes\circ(f,g)$ is a balanced product. For example, for variables $m,m'\!:\!M,n\!:\!N$, $$\vDash f(m+m')\otimes g(n)=(f(m)+f(m'))\otimes g(n)=f(m)\otimes g(n)+f(m')\otimes g(n).$$
    So, there is a unique homomorphism $f\otimes g:M\otimes_kN\rightarrow P\otimes_kR$ such that $(f\otimes g)\circ\otimes=\otimes\circ(f,g)$. 
\end{proof}

\begin{corollary}
    Let $M,P$ be right $A$-modules, $N,Q$ left $A$-modules for some $A\in\text{Ring}(\E)$. Given $A$-module homomorphisms $f:M\rightarrow P,g:N\rightarrow Q$, we have that that $$\vDash\forall m\!:\!M,n\!:\!N,(f\otimes g)(m\otimes n)=f(m)\otimes g(n).$$ 
\end{corollary}

While the above results are sufficient to work fluently with variables of tensor product type, if $\E$ is a \emph{geometric category} (for example, a Grothendeick topos) then we can make explicit the logical structure of the tensor product. Let $\varphi(x\!:\!M\otimes_AN)$ denote the formula $$\bigvee\!\!\!\!\!\phantom{.}_{\phantom{.}_{p\geq0}}\exists m_1\!:\!M,...,m_p\!:\!M,n_1\!:\!N,...,n_p\!:\!N(x=\sum_{i=1}^pm_i\otimes n_i).$$

\begin{lemma}
     Let $\E$ be a geometric category. Subobject $L=\llbracket\varphi\rrbracket$ is a subgroup of $M\otimes_AN$. 
\end{lemma}

\begin{proof}
    Here, the sum notation denotes precisely the group structure on $M\otimes_AN$. So $\vDash\forall x,x'\!:\!M\otimes_AN\Big(${\large$\bigvee$}\!\!\!\!\!$\phantom{.}_{\phantom{.}_{p\geq0}}${\large$\bigvee$}\!\!\!\!\!$\phantom{.}_{\phantom{.}_{q\geq0}}\exists m_1\!:\!M,...,m_p,m_{p+1},...,m_{p+q}\!:\!M,n_1\!:\!N,...,n_p,n_{p+1},...,n_{p+q}\!:\!N((x=\sum_{i=1}^pm_i\otimes n_i)\land(x'=\sum_{i=p+1}^{p+q}m_i\otimes n_i)\land(x+x'=\sum_{i=1}^{p+q}m_i\otimes n_i))\Big)$. Hence, $x+x'\!:\!M\otimes_AN$. 
\end{proof}

\begin{proposition}
    Let $\E$ be a geometric category. For $x\!:\!M\otimes_AN$, we have $\vDash\varphi$. Hence, $M\otimes_AN\cong\llbracket\varphi\rrbracket$.\label{variables-type-tensor-II}
\end{proposition}

\begin{proof}
    $L=\llbracket\varphi\rrbracket$ is a subgroup of $M\otimes_AN$. Let $Q=(M\otimes_AN)/L$ with $q$ the quotient map to $Q$. We have: $0=q\circ\otimes$  as well as $0=0\circ\otimes$. $0$ is a balanced product, and both $0,q$ are group homomorphisms. Hence, by the uniqueness in the universal property, $q=0$, and so $L=M\otimes_AN$.
\end{proof}

\section{Free constructions with a natural number object}

For this section, we take $\E$ to be a topos with natural number object. 

\subsection{Free \texorpdfstring{$A$}{A}-module}

The forgetful functor $\bar{(-)}:A\M\rightarrow\E$ has a \emph{free} left adjoint:

\[\begin{tikzpicture}[scale=1]
	\node (A1) at (0,0) {$A\M$};
	\node (B1) at (1.25,0) {\rotatebox[origin=c]{-90}{$\vdash$}};
	\node (C1) at (2.5,0) {$\E\textcolor{white}{-Mod}$};
	\draw[->,bend left=20]
	(A1) edge node [above] {$\bar{(-)}$} (C1)
	(C1) edge node [below] {$A(-)$} (A1);
\end{tikzpicture}\]
This can be expressed as the following universal property: 
\begin{proposition}[Universal property of the free $A$-module]
    Given map $f:X\rightarrow M$ for some $X\in\E$ and some $A$-module $M$, there exists unique $A$-module homomorphism $\hat{f}:AX\rightarrow M$ s.t. $f=\hat{f}\circ\eta_X$ as morphisms in $\E$. 
\end{proposition}

It is worth avoiding excessive precision. Considering $X\in\E,M\in A\M$ as objects in $\E$, $M$ and $\bar{M}$ are the same object, as are $\bar{AX}$, $AX$ and $A\bar{X}$. So the unit \& counit of our adjunction can be considered as morphisms $\eta_X:X\rightarrow AX$ and $\epsilon_M:AM\rightarrow M$ in $\E$, with $\epsilon_M$ also being a $A$-module homomorphism. From the faithfulness of the forgetful functor, $\epsilon$ is a pointwise epimorphism. Similarly, from \cite{adjunction}, the adjunction sends $f\in\Hom_A(AX,M)$ to the composite $X\xrightarrow{\eta_X}AX\xrightarrow{f}M$; and $g\in\E(X,M)$ to the composite $AX\xrightarrow{Ag}AM\xrightarrow{\epsilon_M}M$, again a $A$-module homomorphism. 

As such, a number of results simplify. For example, any given $g:X\rightarrow\bar{M}$ is equal to the composite $X\xrightarrow{\eta_X}\bar{AX}\xrightarrow{\bar{Ag}}\bar{A\bar{M}}\xrightarrow{\bar{\epsilon_M}}\bar{M}$ \cite{adjunction}; but as an equality of morphisms in $\E$, we can consider this a commutative diagram in $\E$, 
\[\begin{tikzpicture}[scale=2]
    \node (A1) at (0,0.8) {$AX$}; 
    \node (A2) at (0,0) {$X$};
    \node (B1) at (1,0.8) {$AM$};
    \node (B2) at (1,0) {$M$};
    \draw[->]
        (A2) edge node [left] {$\eta_X$} (A1)
        (A2) edge node [above] {$g$} (B2)
        (A1) edge node [above] {$Ag$} (B1);
    \draw[->>]
        (B1) edge node [right] {$\epsilon_M$} (B2);
\end{tikzpicture}\]From this thinking, we obtain the following results:

\begin{proposition}
    Given $M\in A\M$, $\epsilon_M\circ\eta_M=\id_M$ as morphisms in $\E$.\label{free-unit-counit}
\end{proposition}

\begin{proof}
    Take $g=\id_M$ in the above diagram. 
\end{proof}

\begin{proposition}
    The unit $\eta$ is a pointwise monomorphism, and the counit $\epsilon$ is a pointwise epimorphism.\label{unit-mono}
\end{proposition}

\begin{proof}
    As stated, that $\epsilon$ is an epimorphism is an immediate consequence of the forgetful functor being faithful. But from Proposition \ref{free-unit-counit}, $\id=\epsilon\circ\eta$. As $\id$ is pointwise monic, we have that $\eta$ must be pointwise monic. 
\end{proof}
    
Given a variable $x\!:\!X$, when the context is clear we will refer to $x\!:\!AX$ as the variable $\eta_X(x)$ (although in some of our arguments will need to make explicit the presence of $\eta_X$). We prove a free analogue of \handbooktwo (a), showing that maps on the free module are ``determined by their generators": 

\begin{proposition}
    \label{handbooktwo-free} Given $A$-homomorphisms $f,g:AX\rightarrow M$, we have $$f=g\text{ if and only if }\vDash\forall x\!:\!X(f(x)=g(x)).$$
\end{proposition}

\begin{proof}
    Writing precisely, we aim to prove that $\vDash f=g$ iff $f\circ\eta_X=g\circ\eta_X$, whenever $f$ and $g$ are $A$-homomorphisms. We have $\E(X,M)\cong\Hom_A(X,M)$. If $f=g$, then they correspond to the same element of $\E(X,M)$; hence, $f\circ\eta_X=g\circ\eta_X$. If we have that $f\circ\eta_X=g\circ\eta_X$ as functions $\E(X,M)$, then they correspond to the same element of $\Hom_A(X,M)$; hence, $f=g$. 
\end{proof}

Now, $AX$ is an $A$-module object of $\E$, so has addition and $A$-action maps. If $\E$ is a geometric category, we can consider subobject $\llbracket\varphi\rrbracket\xmono{\imath}AX$, where $\varphi(x\!:\!AX)$ denotes the formula $$\bigvee\!\!\!\!\!\phantom{.}_{\phantom{.}_{p\geq0}}\exists a_1,...,a_p\!:\!A,\exists x_1,...,x_p\!:\!X(x=\sum_{i=1}^pa_ix_i).$$ As $\im(\eta_X)\mono\llbracket\varphi\rrbracket$, from the thinking above we can consider instead $X\mono\llbracket\varphi\rrbracket$ a subobject in $\E$. Hence, $\eta_X$ factors through $\imath$ from $X\rightarrow AX$. 

\begin{lemma}
    If topos $\E$ is a geometric category, subobject $\llbracket\varphi\rrbracket$ is an $A$-module object in $\E$, so $\llbracket\varphi\rrbracket\subset AX$ in $A\M$.
\end{lemma}

\begin{proof}
    Write $\imath:\llbracket\varphi\rrbracket\rightarrow AX$, and consider variables $x,x'\!:\!X,a\!:\!A$. Given $\llbracket\varphi\rrbracket\times\llbracket\varphi\rrbracket$ is a subobject of $AX\times AX$, consider the restriction of $AX$'s addition, $+:\llbracket\varphi\rrbracket\times\llbracket\varphi\rrbracket\rightarrow AX$. Then we want $\vDash\forall x,x'\!:\!\llbracket\varphi\rrbracket(\varphi(x+x'))$. But $$\vDash\bigvee\!\!\!\!\!\phantom{.}_{\phantom{.}_{p\geq0}}\bigvee\!\!\!\!\!\phantom{.}_{\phantom{.}_{q\geq0}}\exists a_1,...,a_{p+q}\!:\!A,\exists x_1,...,x_{p+q}\!:\!X,\big(x=\sum_{i=1}^pa_ix_i\land x'=\sum_{i=p+1}^{p+q}a_ix_i\big).$$ As a result, $$\vDash\bigvee\!\!\!\!\!\phantom{.}_{\phantom{.}_{p+q\geq0}}\exists a_1,...,a_{p+q}\!:\!A,\exists x_1,...,x_{p+q}\!:\!X,x+x'=\sum_{i=1}^{p+q}a_ix_i;$$ substituting a $P=p+q$, we see that $\vDash\varphi(x+x')$. Note that this holds even if the $x_i$ variables overlap - the sum notation simply denotes the addition structure of $AX$. 

    Now, given $A\times\llbracket\varphi\rrbracket$ is a subobject of $A\times AX$, consider the restriction of the $A$-module action on $AX$, $A\times\llbracket\varphi\rrbracket\rightarrow AX$. Then we want $\vDash\forall x\!:\!\llbracket\varphi\rrbracket,\forall a\!:\!A(\varphi(ax))$. But $$\vDash\bigvee\!\!\!\!\!\phantom{.}_{\phantom{.}_{p\geq0}}\exists a_1,...,a_p\!:\!A,x_1,...,x_p\!:\!X,x=\sum_{i=1}^pa_ix_i;$$ as a result, $$\vDash\bigvee\!\!\!\!\!\phantom{.}_{\phantom{.}_{p\geq0}}\exists a_1,...,a_p\!:\!A,x_1,...,x_p\!:\!X,ax=\sum_{i=1}^p(aa_i)x_i.$$ Hence, $\vDash\varphi(ax)$. 
\end{proof}

\begin{proposition}
    Let our topos $\E$ be a geometric category. Considered in $A\M$, $\llbracket\varphi\rrbracket=AX$. Hence, writing $x\!:\!AX$,
    $$\vDash\bigvee\!\!\!\!\!\phantom{.}_{\phantom{.}_{p\geq0}}\exists a_1,...,a_p\!:\!A,x_1,...,x_p\!:\!X(x=\sum_{i=1}^pa_ix_i).$$ 
\end{proposition}

\begin{proof}    
    We wish to show $\Hom_A(\llbracket\varphi\rrbracket,M)\cong\E(X,M)$. So, given $f\in\Hom_A(\llbracket\varphi\rrbracket,M)$, let our new adjunction map $f$ to the upper path of the diagram whose left hand triangle commutes:
    \[\begin{tikzpicture}[scale=2]
        \node (A1) at (0,0.5) {$X$};
        \node (B1) at (1,0.5) {$\llbracket\varphi\rrbracket$};
        \node (B2) at (1,0) {$AX$};
        \node (C1) at (2,0.5) {$M$.};
        \draw[->]
        (A1) edge node [above] {} (B1)
        (A1) edge[bend right=30] node [below] {$\eta_X$} (B2)
        (B1) edge node [above] {$f$} (C1)
        (B1) edge node [right] {$\imath$} (B2);
    \end{tikzpicture}\]
    Similarly, given $g\in\E(X,M)$, we identify it with $\hat{g}$, the composite $$\llbracket\varphi\rrbracket\xmono{\imath}kX\xrightarrow{Ag}AM\xrightarrow{\epsilon_M}M.$$
    We turn to internal logic. We have that, given variables $a_1,...,a_p\!:\!A, x_1,...,x_p\!:\!X$ for some $p$, $$\vDash\hat{g}(\sum_{i=1}^pa_i\eta_X(x_i))=\epsilon_M(Ag(\sum_{i=1}^pa_i\eta_X(x_i)))=\sum_{i=1}^pa_i\epsilon_M(Ag(\eta_X(x_i)))=\sum_{i=1}^pa_ig(x_i),$$ from the commutative diagram earlier in the proof. Hence, $$\vDash\bigvee\!\!\!\!\!\phantom{.}_{\phantom{.}_{p\geq0}}\exists a_1,...,a_p\!:\!A,\exists x_1,...,x_p\!:\!X,\Big(x=\sum_{i=1}^pa_ix_i\:\land\:\hat{g}(x)=\sum_{i=1}^pa_ig(x_i)\Big).$$
    In other words, $\hat{g}$ acts on elements of $X$ as $g$ does, and extends linearly, as in the classical case. Sending $\E$-morphism $g:X\rightarrow M$ there and back via our new correspondence obtains composite $g':X\xrightarrow{\eta_X}\llbracket\varphi\rrbracket\xrightarrow{\hat{g}}M$. Then for $x\!:\!X$ we have $\vDash g'(x)=g(x)$, i.e. $g'=g$. Sending $f:\llbracket\varphi\rrbracket\rightarrow M$ there and back via our new correspondence obtains morphism $f':\llbracket\varphi\rrbracket\rightarrow M$ s.t. $$\vDash f'(\sum_{i=1}^pa_i(x_i))=\sum_{i=1}^pa_if(x_i)=f(\sum_{i=1}^pa_i(x_i)),$$ as $f$ is a $A$-module homomorphism. Hence, $$\vDash\bigvee\!\!\!\!\!\phantom{.}_{\phantom{.}_{p\geq0}}\exists a_1,...,a_p\!:\!A,x_1,...,x_p\!:\!X,\Big(x=\sum_{i=1}^pa_ix_i       \:\land\:f'(x)=f(x)\Big).$$
    This simplifies to $\vDash f'(x)=f(x)$, and so $f'=f$ by \cite{handbook} Proposition 6.10.2(a). So, for all $X\in\E,M\in A\M$, we have $\E(X,M)\cong\Hom_A(\llbracket\varphi\rrbracket,M)$; hence, $\llbracket\varphi\rrbracket=AX$. From this, we obtain the naturality of the isomorphism. 
\end{proof}

\subsection{Submodule generated by a subobject}

\begin{definition}
    Suppose $i:X\mono M$ is a monic in $\E$ for $A$-module $M$. The \emph{submodule $\langle X\rangle$ generated by $X$} is the image of the composite $AX\xrightarrow{Ai}AM\xrightarrow{\epsilon_M}M$. 
\end{definition}

\begin{corollary}
    Subobject $\langle X\rangle$ is a submodule of $M$.
\end{corollary}

\begin{proof}
    This follows from the fact that $AX\rightarrow AM\xrightarrow{\epsilon_M}M$ is an $A$-module homomorphism.   
\end{proof}

\begin{proposition}
    \label{handbooktwo-generated}Given $A$-homomorphisms $f,g:\langle X\rangle\to N$, we have $$f=g\text{ if and only if }\vDash\forall x\!:\!X(f(x)=g(x)).$$
\end{proposition}

\begin{proof}
    Here, the variable $x$ is the variable $e(\eta_X(x))\!:\!\langle X\rangle$, where $\eta_X:X\to AX$ is the counit of the free adjunction, and $e:AX\to\langle X\rangle$ is the map into the image of $\epsilon_M\circ A_i$ given by the universal property. It is known that, given a category $\mathscr{C}$ with equalisers, such a morphism is an epimorphism (see, for example, \cite{mitchell}). So, $f=g$ if and only if $f\circ e=g\circ e$. From Proposition \ref{handbooktwo-free}, we have that $f\circ e=g\circ e$ if and only if $\vDash\forall x\!:\!X(f(e(\eta_X(x)))=g(e(\eta_X(x))))$, and so we have our result. 
\end{proof}

\begin{lemma}
    Considering $X\xmono{i}M$ as a subobject of $AM$ (via $\eta_M\circ i$), the submodule $\langle X\rangle$ of $AM$ generated by $X$ is the image of $Ai$. 
\end{lemma}

\begin{proof}
    \label{submodule-generated-free}The subobject $\langle X\rangle$ of $AM$ generated by $X$ is the image of the composite $AX\xmono{Ai}AM\xrightarrow{\eta_{AM}}AAM\xrightarrow{\epsilon_{AM}}AM$; but by Lemma \ref{free-unit-counit}, $\epsilon_{AM}\circ\eta_{AM}=\id_{AM}$, so this composite is simply $Ai$.
\end{proof}

\begin{proposition}
    \label{disjunction-generated}Let our topos $\E$ be a geometric category. Writing $\varphi(m\!:\!M)$ to be $\bigvee\!\!\!\!\!\phantom{.}_{\phantom{.}_{p\geq0}}\exists a_1,...,a_p\!:\!A,x_1,...,x_p\!:\!X(m=\sum_{i=1}^pa_ix_i)$, we have $$\langle X\rangle\cong\llbracket\varphi\rrbracket.$$
\end{proposition}

\begin{proof}
    Write $i:X\mono M$. We have that $\langle X\rangle=\{m\!:\!M|\exists x\!:\!AX(\epsilon_M(Ai(x))=m)\}$. Writing variable $x\!:\!AX$, we also have $\vDash\bigvee\!\!\!\!\!\phantom{.}_{\phantom{.}_{p\geq0}}\exists a_1,...,a_p\!:\!A,x_1,...,x_p\!:\!X(x=\sum_{i=1}^pa_i\eta_Xx_i)$. 
    
    Fix a natural number $p$. For variables $a_1,...,a_p\!:\!A,x_1,...,x_p\!:\!X$, 
    \begin{align}
        \vDash\big((x=\sum_{i=1}^pa_i\eta_Xx_i)\big)\Rightarrow(\epsilon_M(Ai(x))=\epsilon_M(Ai(\sum_{i=1}^pa_i\eta_Xx_i))&=\sum_{i=1}^pa_i(\epsilon_M\circ Ai\circ\eta_X)(x_i)\notag\\&=\sum_{i=1}^pa_ii(x_i))\notag
    \end{align} which we can shorten to $\vDash(x=\sum_{i=1}^pa_i\eta_Xx_i)\Rightarrow(\epsilon_M(Ai(x))=\sum_{i=1}^pa_ii(x_i))$. Hence, by \textbf{T55}, we have $$\vDash\exists a_1,...,a_p\!:\!A,x_1,...,x_p\!:\!X((x=\sum_{i=1}^pa_i\eta_Xx_i)\Rightarrow(\epsilon_M(Ai(x))=\sum_{i=1}^pa_ii(x_i))).$$ 
    
    So, $\vDash\bigvee\!\!\!\!\!\phantom{.}_{\phantom{.}_{p\geq0}}\exists a_1,...,a_p\!:\!A,x_1,...,x_p\!:\!X\big(\epsilon_M(Ai(x))=\sum_{i=1}^pa_ii(x_i))\big)$. Now, for $m\!:\!M$, clearly the formula $\exists x\!:\!AX(\epsilon_M(Ai(x))=m$ is equivalent to the formula $\bigvee\!\!\!\!\!\phantom{.}_{\phantom{.}_{p\geq0}}\exists a_1,...,a_p\!:\!A,x_1,...,x_p\!:\!X\big(\exists x\!:\!AX(\epsilon_M(Ai(x))=m\big)$. Hence, 
    \begin{align}
        \langle X\rangle&=\{m\!:\!M|\exists x\!:\!AX(m=\epsilon_M(Ai(x)))\}\notag\\
        &=\{m\!:\!M|\bigvee\!\!\!\!\!\phantom{.}_{\phantom{.}_{p\geq0}}\exists a_1,...,a_p\!:\!A,x_1,...,x_p\!:\!X\big(\exists x\!:\!AX(m=\epsilon_M(Ai(x))\big)\}\notag\\
        &=\{m\!:\!M|\bigvee\!\!\!\!\!\phantom{.}_{\phantom{.}_{p\geq0}}\exists a_1,...,a_p\!:\!A,x_1,...,x_p\!:\!X\big(\exists x\!:\!AX\big((x=\sum_{i=1}^pa_i\eta_Xx_i)\land(m=\epsilon_M(Ai(x))\big)\big)\}\notag\\
        &=\{m\!:\!M|\bigvee\!\!\!\!\!\phantom{.}_{\phantom{.}_{p\geq0}}\exists a_1,...,a_p\!:\!A,x_1,...,x_p\!:\!X\big(\exists x\!:\!AX\big((x=\sum_{i=1}^pa_i\eta_Xx_i)\land(m=\sum_{i=1}^pa_ii(x_i))\big)\big)\}\notag\\
        &=\{m\!:\!M|\bigvee\!\!\!\!\!\phantom{.}_{\phantom{.}_{p\geq0}}\exists a_1,...,a_p\!:\!A,x_1,...,x_p\!:\!X\big(\exists x\!:\!AX(m=\sum_{i=1}^pa_ix_i)\big)\},\notag
    \end{align} omitting the monic $i$ in our notation. Here we have used that, if $\vDash\forall x\varphi(x)$, then $\{a|\exists x\psi\}=\{a|\exists x(\psi\land\varphi)\}$ for any variable $a$ of some type $A$. 
\end{proof}




\subsection{Constructing the tensor product as a quotient of the free module}

We begin with the ideas of \cite{lang}. Fix $A\in\text{Ring}(A)$. Let $M$ be a right $A$-module, $N$ a left $A$-module. Define the following maps in $\E$:
\begin{itemize} 
    \item $a:M\times A\times N\xrightarrow{\mu_M}M\times N$; 
    \item $b:M\times A\times N\xrightarrow{\mu_N}M\times N$; 
    \item $c:M\times A\times N\cong(M\times N)\times A\xrightarrow{\mu_{M\times N}}M\times N$;
    \item $d:M\times M\times N\xrightarrow{+_M}M\times N$; 
    \item $e:M\times M\times N\xrightarrow{(\id_M^2,\Delta)}M\times M\times N\times N\cong M\times N\times M\times N\xrightarrow{+_{M\times N}}M\times N$; 
    \item $f:M\times N\times N\xrightarrow{+_N}M\times N$; 
    \item $g:M\times N\times N\xrightarrow{(\Delta,\id_M^2)}M\times M\times N\times N\cong M\times N\times M\times N\xrightarrow{+_{M\times N}}M\times N$. 
\end{itemize} 

We can form four morphisms $a-c,b-c:M\times A\times N\xrightarrow{\mu_M}M\times N,d-e:M\times M\times N\xrightarrow{+_M}M\times N,f-g:M\times N\times N\xrightarrow{+_M}M\times N$. Let $S=\im(a-c)\cup\im(b-c)\cup\im(d-e)\cup\im(f-g)$. We can consider $S$ as a subobject of the free module $A(M\times N)$ (by composing $\eta_{M\times N}$ with the inclusion $S\xmono{i}M\times N$). Hence, we can form $A$-submodule $\langle S\rangle$ of $A(M\times N)$ generated by $S$ (from Lemma \ref{submodule-generated-free}, $\langle S\rangle=\im(Ai)$). 

From \cite{handbook} Proposition 6.10.3, 
\begin{align}
    S&=\im(a-c)\cup\im(b-c)\cup\im(d-e)\cup\im(f-g)\notag\\&=\{x\,|\,x\in\im(a-c)\lor x\in\im(b-c)\lor x\in\im(d-e)\lor x\in\im(f-g)\}\notag\\&=\{x\,|\,\exists m\exists n\exists a(x=(am,n)-a(m,n))\lor\,\exists m\exists n\exists a(x=(m,an)-a(m,n))\notag\\
    &\lor\exists m,m'\exists n(x\!=\!(m+m',n)\!-\!(m,n)\!-\!(m',n))\lor\exists m\exists n,n'(x\!=\!(m,n+n')\!-\!(m,n)\!-\!(m,n'))\}.\notag
\end{align}

\begin{proposition}
    Given $M,N\in A\M$, we have $$M\otimes_AN\cong A(M\times N)/\langle S\rangle,$$ with $\otimes$ as the composite $M\times N\xmono{\eta_{M\times N}}A(M\times N)\epi M\otimes_AN$. 
\end{proposition}

\begin{proof}
    We follow \cite{conrad}. Suppose we have balanced product $M\times N\xrightarrow{f}P$. Viewing it as a morphism in $\E$, the free-forgetful adjunction gives us a corresponding $A$-module homomorphism $\hat{f}:A(M\times N)\rightarrow P$, s.t. $\vDash\hat{f}(\eta_{M\times N}(m,n))=f(m,n)$, i.e. $\hat{f}\circ\eta_{M\times N}=f$.

    To argue that $\hat{f}$ is zero on $S$, we need only show that $f$ is zero on $S$, by Proposition \ref{handbooktwo-generated}. We can use \cite{handbook} Proposition 6.8.5 (T66) to pull out the existential quantifiers of the above statement of $S$ to obtain, for variable $x\!:\!S$,
    \begin{align}
        \vDash\exists m,m'\exists n,n'\exists a\big(&x=(am,n)-a(m,n))\lor(x=(m,an)-a(m,n))\notag\\&\lor(x=(m+m',n)-(m,n)-(m',n))\notag\\&\lor(x=(m,n+n')-(m,n)-(m,n'))\big).\notag
    \end{align} Applying $f$ to a variable $x\!:\!S$ in the above expression yields zero,  so we obtain
    $$\vDash\exists m,m'\exists n,n'\exists a\big(f(x)=0\lor f(x)=0\lor f(x)=0\lor f(x)=0\big),$$ so $\vDash\exists m,m'\exists n,n'\exists a(f(x)=0)$. As $m,n',n,n',a$ aren't variables of the formula $f(x)=0$, this implies $\vDash f(x)=0$, by \cite{handbook} Proposition 6.8.1 (T49). So $\hat{f}$ is zero on $S$ as a subgroup of $A(M\times N)$. By Proposition \ref{handbooktwo-generated}, it is zero on $\langle S\rangle$. So $\hat{f}$ factors through $q:A(M\times N)\rightarrow A(M\times N)/\langle S\rangle$. So we have $\hat{f}':M\otimes_AN\rightarrow P$ such that $\hat{f}'\circ q=\hat{f}$, so $\hat{f}'\circ q\circ\eta_{M\times N}=\hat{f}\circ\eta_{M\times N}$; in other words, $\hat{f}'\circ\otimes=f$. Uniqueness comes from Proposition \ref{handbooktwo-tensor}.
\end{proof}

We provide an alternative proof of Proposition \ref{variables-type-tensor-II}, given our topos $\E$ is a geometric category: 

\begin{proof}
    As $q$ is an epimorphism, we have $\vDash\forall x\!:\!M\otimes_AN(\exists\tilde{x}\!:\!A(M\times N)(x=q(\tilde{x})))$. This implies $\vDash\forall x\!:\!M\otimes_AN\big(\bigvee\!\!\!\!\!\phantom{.}_{\phantom{.}_{p\geq0}}\exists m_1,...,m_p\!:\!M,\exists n_1,...,n_p\!:\!N(x=q(\sum_{i=1}^pa_i(m_i,n_i))=\sum_{i=1}^pa_iq(m_i,n_i)=\sum_{i=1}^pa_i(m_i\otimes n_i)=\sum_{i=1}^p(a_im_i)\otimes n_i)\big)$, as $q$ is a $A$-module homomorphism. Given $a_im_i\!:\!M$ for all $a_i\!:\!A,m_i\!:\!M$, we have our result. 
\end{proof}

\subsection{Constructing the tensor product as a coequaliser}

Here, we follow \cite{lang} more closely. Let $M,N\in A\M$. Recall the 7 maps $a,b,c,d,e,f,g$ defined in the previous subsection. These induce maps $Aa,Ab,Ac:A(M\times A\times N)\rightarrow A(M\times N)$; $Ad,Ae:A(M\times M\times N)\rightarrow A(M\times N)$; $Af,Ag:A(M\times N\times N)\rightarrow A(M\times N)$. Form the diagram $D$ in $A\M$ with these four objects and seven morphisms.

\begin{proposition}
    We have $M\otimes_AN\cong\emph{colim}D$, i.e. the colimit
    \[\begin{tikzpicture}[scale=2]
	\node (A2) at (0,0.7) {$A(M\times A\times N)$};
        \node (B1) at (1,1.4) {$A(M\times M\times N)$};
        \node (B3) at (1,0) {$A(M\times N\times N)$};
	\node (C2) at (2,0.7) {$A(M\times N)$};
	\node (D2) at (3.2,0.7) {$M\otimes_AN$};
	\draw[->]
	(C2) edge node [above] {$q$} (D2)
        (A2) edge node [above] {$Aa,Ab,Ac$} (C2);
    \begin{scope}[transform canvas={yshift=.4em}]
        \draw [->] (A2) edge node [above] {} (C2);
    \end{scope}
    \begin{scope}[transform canvas={yshift=-.4em}]
        \draw [->] (A2) edge node [below] {} (C2);
    \end{scope}
    \begin{scope}[transform canvas={xshift=1.0em}]
        \draw [->] (B1) edge node [above right] {$Ad,Ae$} (C2);
    \end{scope}
    \begin{scope}[transform canvas={xshift=0em}]
        \draw [->] (B1) edge node [below left] {} (C2);
    \end{scope}
    \begin{scope}[transform canvas={xshift=1.0em}]
        \draw [->] (B3) edge node [below right] {$Af,Ag$} (C2);
    \end{scope}
    \begin{scope}[transform canvas={xshift=0em}]
        \draw [->] (B3) edge node [above left] {} (C2);
    \end{scope}
\end{tikzpicture}\]
\end{proposition}

\begin{proof}
    From Proposition \ref{tensor-identities}, we know that $\otimes\circ d=\otimes\circ e$, i.e. $q\circ\eta_{M\times N}\circ d=q\circ\eta_{M\times N}\circ e$; hence, by the naturality of $\eta$, $q\circ Ad\circ\eta_{M^2\times N}=q\circ Ae\circ\eta_{M^2\times N}$. From Proposition \ref{handbooktwo-free}, $q\circ Ad=q\circ Ae$. A similar argument shows that $q$ equalises $Af$ and $Ag$, and each pair of $Aa,Ab,Ac$. 
    
    Suppose there exists some object $T'$ with $A$-module homomorphism $A(M\times N)\xrightarrow{q'}T'$ such that $q'$ equalises our maps. Then by the free-forgetful adjunction, $q'\in\Hom_A(A(M\times N),T')$ induces $\hat{q'}=q'\circ\eta_{M\times N}\in\E(M\times N,T')$. But from the fact that $q'$ equalises $Ad$ and $Ae$, it must equalise $Ad\circ\eta_{M^2\times N},Ae\circ\eta_{M^2\times N}$; hence, it equalises $\eta_{M\times N}\circ d,\eta_{M\times N}\circ e$. By similar arguments, it equalises $\eta_{M\times N}\circ a,\eta_{M\times N}\circ b$ and $\eta_{M\times N}\circ f,\eta_{M\times N}\circ g$. Therefore, $q'\circ\eta_{M\times N}$ is a balanced product. Hence, $\hat{q'}=q'\circ\eta_{M\times N}$ induces unique map $M\otimes_AN\rightarrow T'$ which makes our diagram commute.   
\end{proof}

\section{Algebra over a ring object}

\subsection{Internal hom and tensor product}

It is a simple exercise to check that $A$-action on a $A$-module $M$ is a balanced product in $A$ and $M$ whenever $A$ is commutative. Therefore, we also have the following result which shows that any $A$-action on a $A$-module factors through the tensor:

\begin{lemma}
    Let $M\in A\M$, so that $M$ is equipped with $A$-action $\mu_M:A\times M\rightarrow M$. Then we have action $\hat{\mu}_M:A\otimes_AM\rightarrow M$ so that $\hat{\mu}_M\circ\otimes=\mu_M$. 
\end{lemma}



\begin{lemma}
    For all $M\in\text{Mod-}A$, we have $M\otimes_AA\cong M$.
\end{lemma}

\begin{proof}
    We have that $\Hom_A(M\otimes_AA,M)\cong\Hom_A(M,[A,M]_A)\cong\Hom_A(M,M)$. By the naturality of these isomorphisms, we have that $M\otimes_AA\cong M$. 
\end{proof}

\subsection{Bimodules over an algebra}

Let $k$ be a commutative ring object in $\E$, and $A$ a (not necessarily commutative) ring object in $k\M$, i.e. $A\in k$-Alg. Write $A$-Mod, $A$-Bimod as the categories of left $A$-modules and $A$-bimodules, respectively. 

\begin{lemma}
    Given $A$-module $M$ and $k$-module $N$, we have $M\otimes_kN\in A\M$. If $M$ is an $A$-bimodule, then $M\otimes_kN\in A\text{-Bimod}$.
\end{lemma}

\begin{proof}
    We let the $A$-module action on $M\otimes_kN$ be the composite $A\otimes_k(M\otimes_kN)\cong(A\otimes_kM)\otimes_kN\xrightarrow{\mu_M\otimes\id_N}M\otimes_kN$, so that $\vDash a(m\otimes n)=(am)\otimes n$. From here, it is a simple verification that this is in fact an $A$-module action. If $M\in A\text{-Bimod}$, we construct the left action as above, and take the appropriate isomorphisms $(M\otimes_kN)\otimes_kA\cong M\otimes_k(N\otimes_kA)\cong M\otimes_k(A\otimes_kN)\cong (M\otimes_kA)\otimes_kN$ for the right action. The result follows. 
\end{proof}

\begin{lemma}
Given $M,N\in A$-Mod, $[M,N]_k\in A\M$. If $N$ is an $A$-bimodule, then $[M,N]_k\in A$-Bimod. 
\end{lemma}

\begin{proof}
    We have $\Hom_A(A\otimes_A[M,N]_k,[M,N]_k)\cong\Hom_k(A\otimes_A[M,N]_k\otimes_AM,N)$. The evaluation map $\ev_M:[M,N]_k\otimes_AM\rightarrow N$ is a morphism in both $k\M$ and $A\M$. Hence, we let our $A$-module action on $[M,N]_k$ be the adjunct of $A\otimes_A[M,N]_k\otimes_AM\xrightarrow{\id_A\otimes\ev_M}A\otimes_AN\xrightarrow{\mu_N}A$, so that $\vDash(af)(m)=af(m)$. We can verify that this corresponds to an $A$-module action. If $N\in A$-Bimod, we construct the left action as above, and take the isomorphism $A\otimes_AN\cong N\otimes_AA$ for the right action. The result follows. 
\end{proof}

\subsection{Enriched projectivity}

Recall that $M\in k\M$ is projective iff the functor $\Hom_k(M,-)$ is exact. We say that $M$ is \emph{internally projective} if the functor $[M,-]_k$ is exact (Definition \ref{internal-proj}), and that $M$ is \emph{enriched projective} if it is both internally projective and projective (Definition \ref{enriched-proj}).

\begin{proposition}
    If $M,N\in k\M$ are enriched projective, then so is $M\otimes_kN$. \label{tensor-projective}
\end{proposition}

\begin{proof}
    Suppose $M,N$ are \emph{enriched projective} in $k\M$. We want to show that $\Hom_k(M\otimes_kN,-)$ and $[M\otimes_kN,-]_k$ are exact.  

    Let $0\rightarrow A\rightarrow B\rightarrow C\rightarrow 0$ be a short exact sequence of $k$-modules. By the exactness of $[N,-]_k$, we have short exact sequence $$0\to[N,A]_k\to[N,B]_k\to[N,C]_k\to0.$$ By the exactness of $\Hom_k(M,-)$ we have short exact sequence 
\[\begin{tikzpicture}[scale=2]
	\node (A1) at (0,0.5) {$0$};
	\node (A2) at (0,0) {$0$};
	\node (B1) at (1.2,0.5) {$\Hom_k(M,[N,A]_k)$};
	\node (B2) at (1.2,0) {$\Hom_k(M\otimes_kN,A)$};
        \node
    (C1) at (3,0.5) {$\Hom_k(M,[N,B]_k)$};
	\node (C2) at (3,0) {$\Hom_k(M\otimes_kN,B)$};
        \node
    (D1) at (4.8,0.5) {$\Hom_k(M,[N,C]_k)$};
	\node (D2) at (4.8,0) {$\Hom_k(M\otimes_kN,C)$.};
	\draw[->]
    (A1) edge node [above] {} (B1)
	(A2) edge node [below] {} (B2)
	(B1) edge node[sloped] [above] {$\simeq$} (B2)
    (B1) edge node [above] {} (C1)
	(B2) edge node [below] {} (C2)
	(C1) edge node[sloped] [above] {$\simeq$} (C2)
    (C1) edge node [above] {} (D1)
	(C2) edge node [below] {} (D2)
	(D1) edge node[sloped] [above] {$\simeq$} (D2);
\end{tikzpicture}\]
So $\Hom_k(M\otimes_kN,-)$ is right exact, i.e. $M\otimes_kN$ is projective. 

By the exactness of $[M,-]_k$, we have short exact sequence 
    \[\begin{tikzpicture}[scale=2]
	\node (A1) at (0,0.5) {$0$};
	\node (A2) at (0,0) {$0$};
	\node (B1) at (1,0.5) {$[M,[N,A]_k]_k$};
	\node (B2) at (1,0) {$[M\otimes_kN,A]_k$};
        \node
    (C1) at (2.4,0.5) {$[M,[N,B]_k]_k$};
	\node (C2) at (2.4,0) {$[M\otimes_kN,B]_k$};
        \node
    (D1) at (3.8,0.5) {$[M,[N,C]_k]_k$};
	\node (D2) at (3.8,0) {$[M\otimes_kN,C]_k$.};
        \node (E1) at (4.8,0.5) {$0$};
	\node (E2) at (4.8,0) {$0$};
	\draw[->]
    (A1) edge node [above] {} (B1)
	(A2) edge node [below] {} (B2)
	(B1) edge node[sloped] [above] {$\simeq$} (B2)
    (B1) edge node [above] {} (C1)
	(B2) edge node [below] {} (C2)
	(C1) edge node[sloped] [above] {$\simeq$} (C2)
    (C1) edge node [above] {} (D1)
	(C2) edge node [below] {} (D2)
	(D1) edge node[sloped] [above] {$\simeq$} (D2)
    (D1) edge node [above] {} (E1)
	(D2) edge node [below] {} (E2);
\end{tikzpicture}\] So $[M\otimes_kN,-]_k$ is also right exact, i.e. $M\otimes_kN$ is internally projective. 
\end{proof}





\end{document}